\documentclass[letterpaper, 11pt,  reqno]{amsart}

\usepackage{amsmath,amssymb,amscd,amsthm,amsxtra, esint}

\usepackage{mathrsfs} 

\usepackage[implicit=true]{hyperref}

\usepackage{color}
\usepackage{ulem}
\usepackage[makeroom]{cancel}

\allowdisplaybreaks[2]

\definecolor{gr}{rgb}   {0.,   0.69,   0.23 }
\definecolor{bl}{rgb}   {0.,   0.5,   1. }
\definecolor{mg}{rgb}   {0.85,  0.,    0.85}
\definecolor{yl}{rgb}   {0.8,  0.7,   0.}
\definecolor{or}{rgb}  {0.7,0.2,0.2}

\newtheorem{theorem}{Theorem} [section]

\newtheorem{lemma}[theorem]{Lemma}
\newtheorem{proposition}[theorem]{Proposition}
\newtheorem{remark}[theorem]{Remark}

\DeclareMathOperator*{\supp}{supp}

\newcommand{\noi}{\noindent}
\newcommand{\Z}{\mathbb{Z}}
\newcommand{\R}{\mathbb{R}}

\newcommand{\T}{\mathbb{T}}

\let\P= \undefined
\newcommand{\P}{\mathbf{P}}

\newcommand{\E}{\mathbb{E}}

\newcommand{\F}{\mathcal{F}}

\newcommand{\al}{\alpha}
\newcommand{\be}{\beta}
\newcommand{\dl}{\delta}

\newcommand{\nb}{\nabla}

\newcommand{\eps}{\varepsilon}
\newcommand{\kk}{\kappa}
\newcommand{\g}{\gamma}

\newcommand{\s}{\sigma}

\newcommand{\ft}{\widehat}

\newcommand{\wt}{\widetilde}
\newcommand{\cj}{\overline}

\newcommand{\dt}{\partial_t}
\newcommand{\dd}{\partial}

\newcommand{\ta}{\theta}

\renewcommand{\l}{\ell}
\renewcommand{\o}{\omega}
\renewcommand{\O}{\Omega}

\newcommand{\les}{\lesssim}

\newcommand{\jb}[1]
{\langle #1 \rangle}

\renewcommand{\b}{\be}
\newcommand{\ind}{\mathbf 1}

\newcommand{\N}{\mathbb{N}}

\newcommand{\J}{\mathcal{J}}

\newcommand{\EE}{\mathcal{E}}

\renewcommand{\H}{\vec{H}}

\newtheorem*{ackno}{Acknowledgements}

\numberwithin{equation}{section}
\numberwithin{theorem}{section}

\newcommand{\NLW}{\textup{NLW}}

\newcommand{\Q}{\mathbb{Q}}
\newcommand{\PP}{\mathbb{P}}

\newcommand{\HH}{\mathscr{E}}

\newcommand{\C}{\mathcal{C}}

\DeclareMathOperator{\Law}{Law}

\newcommand{\ZZ}{\mathcal{Z}}

\newcommand{\muu}{\vec{\mu}}
\newcommand{\nuu}{\vec{\nu}}

\newcommand{\rhoo}{\vec{\rho}}
\newcommand{\rhooo}{\vec{\wt{\rho}}}

\newcommand{\B}{B}

\newcommand{\hra}{\hookrightarrow}

\newcommand{\W}{\mathcal{W}}
\newcommand{\U}{\mathcal{U}}

\newcommand{\dr}{\theta}
\newcommand{\Dr}{\Theta}

\newcommand{\Ha}{\mathbb{H}_a}
\newcommand{\Hc}{\mathbb{H}_c}

\renewcommand{\u}{\vec{u}}

\newcommand{\DD}{Q_{s, N}( u_N)}
\newcommand{\DDD}{Q_{s}( u)}

\begin{document}
\baselineskip = 14pt

\title[Quasi-invariant Gaussian measures for the 3-$d$ NLW]
{Quasi-invariant Gaussian measures for the nonlinear wave equation in three dimensions}
\author[T.~Gunaratnam, T.~Oh,  N.~Tzvetkov, and H.~Weber]
{Trishen S.~Gunaratnam, Tadahiro Oh, Nikolay Tzvetkov,
and  Hendrik Weber}

\address{
Trishen S.~Gunaratnam\\
Section de Math\'ematiques \\
 Universit\'e de Gen\`eve\\
 Rue du Conseil-G\'en\'eral 7-9\\ 1205 Gen\`eve \\Switzerland}

 \email{trishen.gunaratnam@unige.ch}

\address{
Tadahiro Oh, School of Mathematics\\
The University of Edinburgh\\
and The Maxwell Institute for the Mathematical Sciences\\
James Clerk Maxwell Building\\
The King's Buildings\\
Peter Guthrie Tait Road\\
Edinburgh\\ 
EH9 3FD\\
 United Kingdom}

\email{hiro.oh@ed.ac.uk}

\address{
Nikolay Tzvetkov\\
Universit\'e de Cergy-Pontoise\\
 2, av.~Adolphe Chauvin\\
  95302 Cergy-Pontoise Cedex\\
  France}

\email{nikolay.tzvetkov@u-cergy.fr}

\address{
Hendrik Weber\\
Department of Mathematical Sciences\\
 University of Bath\\
 4 West\\ Claverton Down\\ Bath \\ BA2 7AY\\United Kingdom}

 \email{H.Weber@bath.ac.uk}

\begin{abstract}
We prove quasi-invariance of Gaussian measures
supported on Sobolev spaces under the dynamics of  the three-dimensional
defocusing cubic nonlinear wave equation.  
As in the previous work on  the two-dimensional case, 
we employ
a simultaneous renormalization on the energy functional and its time derivative.
Two new  ingredients in the three-dimensional case
are (i)~the construction of the weighted Gaussian measures, based on a variational formula for the partition function inspired by Barashkov and Gubinelli~(2018),
%
and (ii)~an improved argument in controlling
the growth of the truncated weighted Gaussian measures,  
where we combine a deterministic growth bound
of solutions 
with stochastic estimates on random distributions.
\end{abstract}


\maketitle

\tableofcontents

\section{Introduction}

\subsection{Main result}

We consider
 the following defocusing cubic nonlinear wave equation (NLW) on the three-dimensional torus 
 $\T^3 = (\R/(2\pi \Z))^3$:
\begin{equation}\label{NLW0}
\partial_t^2 u-\Delta u+u^3=0,
\end{equation}

\noi
where $u : \T^3 \times \R\rightarrow \R$ is the unknown function. 
With $v = \dt u$, we rewrite \eqref{NLW0} in the following vectorial form:
\begin{equation}\label{NLW}
\begin{cases}
\partial_t u = v\\
\partial_t v = \Delta u - u^3.
\end{cases}
\end{equation}

\noi
Given $\s \in \R$, let $H^\s(\T^3)$ denote the classical $L^2$-based  Sobolev space of order $\s$
defined by the norm:
\[ \| u\|_{H^\s} = \| \jb{n}^\s\ft u(n)\|_{\l^2(\Z^3)}, \]
 
\noi 
where 
$\jb{\,\cdot\,} = (1+|\,\cdot\,|^2)^\frac{1}{2}$ and $\ft u$ denotes the Fourier transform
of $u$. 
A classical argument yields global well-posedness
of 
the Cauchy problem~\eqref{NLW} 
in the Sobolev spaces:
\[
\H^\s(\T^3)\stackrel{\text{def}}{=} H^\s(\T^3)\times H^{\s-1}(\T^3)
\]
for $\s \geq 1$ and, consequently, admits a global flow $\Phi_{\NLW}$ 
(see Lemma~\ref{LEM:CP_GWP} below) on these spaces. 

Given $s \in \R$, let $\muu_s$ denote the Gaussian measure with Cameron-Martin space $\H^{s+1}(\T^3)$. Denoting $\vec{u} = (u,v)$, the Gaussian measure $\muu_s$ has a formal density:
\begin{align*}
\begin{split}
 d \muu_s 
&   = Z_s^{-1} e^{-\frac 12 \| \vec{u} \|_{{\H}^{s+1}}^2} d\vec{u}
\rule[-4mm]{0mm}{0mm}
\\
& =   \prod_{n \in \Z^3} Z_{s, n}^{-1}
 e^{-\frac 12 \jb{n}^{2(s+1)} |\ft u (n)|^2}   
  e^{-\frac 12 \jb{n}^{2s} |\ft v (n)|^2}   
 d\ft u (n) d\ft v (n).
\end{split}
\end{align*}

\noi
Samples $\vec u^\o = (u^\o, v^\o)$ from $\muu_s$ can be constructed via 
the following Karhunen-Lo\`eve expansions:\footnote{By convention, 
 we endow $\T^3$ with the normalized Lebesgue measure $ (2\pi)^{-3} dx$.}
\begin{equation}\label{series}
u^\o(x) = \sum_{n \in \Z^3} \frac{g_n(\o)}{\jb{n}^{s+1}}e^{in\cdot x}
\qquad\text{and}\qquad
v^\o(x) = \sum_{n \in \Z^3} \frac{h_n(\o)}{\jb{n}^{s}}e^{in\cdot x}, 
\end{equation}

\noi
where  $\{ g_n \}_{n \in \Z^3}$ and $\{ h_n \}_{n \in \Z^3}$
are collections of standard complex-valued Gaussian variables which are independent modulo the condition\footnote{In particular, we impose that $g_0$ and $h_0$ are real-valued.} $g_n = \cj{g_{-n}}$ and $h_n = \cj{h_{-n}}$.
It is easy to see that the series~\eqref{series} converge in $L^2(\Omega;\H^{\s}(\T^3))$ 
 for 
 \begin{align}
 \s < s- \frac12
 \label{sig}
 \end{align}
and therefore the map
\[\o \in \O \longmapsto (u^\o, v^\o)\]

\noi
induces the Gaussian measure  $\muu_s$ 
as a probability measure on $\H^\s(\T^3)$ for the same range of~$\s$.
Our main goal in this paper is to study the transport property
of the Gaussian measure $\muu_s$ under the dynamics of \eqref{NLW}.
We state our main result.

\begin{theorem}\label{THM:NLW}
Let $s\geq 4$ be an even integer. Then,  $\muu_s$ is \textit{quasi-invariant} under the dynamics
of the defocusing cubic NLW \eqref{NLW} on $\T^3$. More precisely, for any $t \in \mathbb{R}$, 
the Gaussian measure $\muu_s$ and its pushforward under $\Phi_{\NLW}(t)$ are mutually absolutely continuous. 
\end{theorem}

Theorem \ref{THM:NLW} ensures the propagation of almost sure properties of $\muu_s$ along the flow. This is important because, in infinite dimensions,
 many interesting properties concerning small-scale behavior under a Gaussian measure
  hold true with probability 0 or 1. This is an implication of Fernique's theorem (Theorem 2.7 in  \cite{DZ});
 under a Gaussian measure, any given norm is finite with probability 0 or 1. 
For example, samples $\u$ of the Gaussian measure $\muu_s$ 
almost surely belong to  the $L^p$-based Sobolev spaces  $\vec W^{\s,  p}(\T^3)$ for any $p \geq 1$ 
and more generally 
to  the 
Besov spaces, $\vec{B}^{\s}_{p,q}(\T^3)$ for any $p , q \geq 1$, 
including the case $p= q =\infty$ (H\"older-Besov space),
provided that $\s$ satisfies \eqref{sig}.
Theorem \ref{THM:NLW} then implies that these $L^p$-based regularities are transported along the nonlinear flow.  An analogous statement for deterministic initial data is expected to  fail in general. 
     See \cite{Lit, Peral, Sogge}.

Theorem \ref{THM:NLW} is an addition to a series of recent results \cite{Tzvet, OTz, OST, OTz3, OTT} that has made significant progress in the study of transport properties of Gaussian measures under nonlinear Hamiltonian PDEs. The general strategy, as introduced by the third author in \cite{Tzvet}, is to study quasi-invariance of the Gaussian measures $\muu_s$ indirectly by studying weighted Gaussian measures, where the weight corresponds to a correction term that arises due to the presence of the nonlinearity. 
See Subsection \ref{SUBSEC:cor}.
The two key steps   in this strategy are (i) the construction of the weighted Gaussian measure
and  (ii) an energy estimate on the time derivative of the modified energy
(that is, the energy of the Gaussian measure plus the correction term). 
In \cite{OTz3}, the second and third authors employed this strategy
and proved  the analogue of Theorem \ref{THM:NLW} in the two-dimensional case.
This was done
by introducing a simultaneous renormalization on the modified energy functional and its time derivative
and then performing
 a delicate analysis centered on a quadrilinear Littlewood-Paley expansion.  

As pointed out in \cite{OTz3}, the argument in the two-dimensional case does not extend
to the current three-dimensional setting.
The proof of Theorem \ref{THM:NLW} uses two new key ingredients. 
The first is the use of a variational formula in constructing weighted Gaussian measures, inspired by Barashkov and Gubinelli 
\cite{BG}.
The second new ingredient appears in studying 
the growth of the truncated weighted Gaussian measures, 
where we combine
 a deterministic growth bound
on solutions (as in a recent paper by Planchon, Visciglia, and the third author \cite{PTV})
with stochastic estimates on random distributions
(as in the two-dimensional case \cite{OTz3}).
This hybrid argument allows us to use a softer energy estimate to prove quasi-invariance.
Our simplification also comes from the use of Besov spaces
in the spirit of \cite{MWX}.
This results in a significantly simpler proof of quasi-invariance in the harder, physically relevant three-dimensional case as compared with the two-dimensional case.

\subsection{Remarks and comments} 
\label{SEC1:2}

(i) 
A slight modification of the proof of Theorem \ref{THM:NLW}
shows that the Gaussian measures $\muu_s$ are also quasi-invariant under the nonlinear Klein-Gordon
equation:
\begin{equation}\label{KG}
\begin{cases}
\partial_t u = v\\
\partial_t v = (\Delta - 1) u - u^3.
\end{cases}
\end{equation}

\noi
It is easy to see that $\muu_s$
is invariant under the linear Klein-Gordon equation,
i.e.~removing $u^3$ in \eqref{KG},
which  trivially implies that almost sure properties of $\muu_s$ are transported along the flow of the linear dynamics.
The addition of a defocusing cubic nonlinearity into the equation destroys invariance 
but the quasi-invariance of $\muu_s$ for \eqref{KG} can be interpreted as saying that the nonlinear flow retains the small-scale properties of the linear flow.

In order to obtain invariance of $\muu_s$ under the linear wave equation, one would need to replace 
$\jb{\,\cdot\,}$ with $|\cdot|$ in \eqref{series}, which would raise an issue at  the zeroth Fourier mode (see Remark \ref{REM:zero_freq}). 
Nevertheless, in the study of small-scale properties of solutions, 
this issue is irrelevant and one can easily show that  $\muu_s$ is quasi-invariant under the linear wave equation.
Theorem~\ref{THM:NLW} then implies 
that the NLW dynamics also retains   the small-scale properties of the linear wave dynamics.

\smallskip
\noi
(ii) The restriction that $s$ is an even integer in Theorem \ref{THM:NLW} 
comes from an application of  the classical Leibniz rule in order 
to derive the right correction term for the modified energy and the weighted Gaussian measure. 
In terms of regularity restrictions, the construction of the weighted Gaussian measure works for any real $s > \frac 32$ 
(Proposition \ref{PROP:QFT0}).
Our argument for the energy estimate 
(Proposition \ref{PROP:Nenergy}) only requires $s > \frac52$ but, in our derivation of a modified energy,
 we also use the classical Leibniz rule for $(-\Delta)^{\frac{s}{2}}$ which only works if $s$ is an even integer. It may be possible to relax this second condition using a fractional Leibniz rule to go below  $s = 4$.
 At present, however, we do not know how to do this.\footnote{In a recent paper \cite{STX},
Theorem \ref{THM:NLW} was extended to the range $s > \frac 52$.
The authors in \cite{STX} also proved  quasi-invariance of $\muu_s$ for the quintic NLW 
(with $u^5$ replacing $u^3$ in \eqref{NLW0}) in the same range $s > \frac 52$.
}

\smallskip
\noi
(iii) Our new hybrid argument in proving Theorem \ref{THM:NLW}
requires a softer energy estimate than that in \cite{OTz3} and is also applicable to the two-dimensional case.
We point out, however, that the argument in~\cite{OTz3}, 
involving heavier multilinear analysis, 
provides better quantitative information on 
the growth of the truncated weighted Gaussian measures.
See Remark \ref{REM:loss}.
For example, the argument in \cite{OTz3} allows us to prove higher $L^p$-integrability
of the Radon-Nikodym derivative of the weighted Gaussian measures
(with an energy cutoff), while our proof of Theorem \ref{THM:NLW}
does not provide such extra information.

\smallskip
\noi
(iv)
It would be of interest
to investigate the quasi-invariance property of $\muu_s$ for NLW with a higher order nonlinearity
or in higher dimensions. Our techniques appear to carry over to higher order nonlinearities. This might even permit to analyze energy-supercritical equations
(such as the three-dimensional septic NLW),
where global well-posedness  is not known. 
Consequently, one might aim to prove ``local-in-time'' quasi-invariance (as stated in~\cite{Bo_proceeding}).\footnote{In \cite{STX}, such local-in-time quasi-invariance of $\muu_s$, $s > 3$,  was shown
for NLW on $\T^3$ with a higher order nonlinearity $u^{2k + 1}$ for an integer $k \geq 3$.}
See also~\cite{PTV}
for an example of a local-in-time quasi-invariance result.
See also Remark \ref{REM:new1} below.

\smallskip
\noi
(v) Quasi-invariance results such as Theorem \ref{THM:NLW} are complimentary to the study of low regularity well-posedness with  random initial data. 
Starting with the seminal work of Bourgain \cite{BO94, BO96}, 
  there has been intensive study on the random data Cauchy theory
for  nonlinear dispersive PDEs (we refer the readers to \cite{BOP4} for a more detailed survey of the literature). There are two related directions in this study.
The first one is the study of 
invariant measures associated with conservation laws such as Gibbs measures, 
in particular,  the construction of almost sure global-in-time dynamics
via the so-called Bourgain's invariant measure argument; 
see \cite{OTz, BOP4} for the references therein. 
The other is the study of almost sure well-posedness
with respect to random initial data.
Here, 
 one can often exploit the higher $L^p_x$-based regularity made accessible by randomization 
 of initial data  to establish well-posedness below critical thresholds, 
 where equations are ill-posed in $L^2$-based Sobolev spaces. 
 In the context of NLW, see the work  \cite{BT08, BT14} by Burq and the third author
 for almost sure local well-posedness.
There are also globalization arguments in this probabilistic setting;
see  \cite{BT14, Poc, OP1, OP2}.

As for  the defocusing cubic NLW \eqref{NLW} on $\T^3$, 
the scaling symmetry induces the critical regularity $\s_\text{crit} = \frac 12$.
It is known that 
 \eqref{NLW} is locally well-posed in $\H^\s(\T^3)$ for $\s \geq \frac 12$, 
while it is  ill-posed for $\s < \frac 12$; see \cite{LS, CCT, BT08, OOTz}.
 In \cite{BT08, BT14}, 
 Burq and the third author proved almost sure global 
 well-posedness of \eqref{NLW} with respect to the random initial data in 
 \eqref{series} for $s > \frac 12$, namely for $\s > 0$.
 In this regime, the flow $\Phi_{\NLW}$ exists almost surely globally in time.
Then, it is natural to ask the following question.

\smallskip

\noi
{\bf Problem.} 
{\it  Study the transport property of the Gaussian measures $\muu_s$
for low values of $s >\frac 12$, in particular in the regime where 
the global-in-time dynamics is constructed only probabilistically.
}

\subsection{Organization}
In Section \ref{SEC:TOOL}, we introduce basic tools in our proof: Besov spaces, 
the Wiener chaos estimate, 
the classical well-posedness theory of \eqref{NLW},  and also deterministic growth bounds. 
In Section \ref{SEC:PLAN}, we present the proof of Theorem \ref{THM:NLW}
assuming (i) the construction of the weighted Gaussian measures 
(Proposition \ref{PROP:QFT0})
and (ii) the energy estimate (Proposition \ref{PROP:Nenergy}).
Section \ref{SEC:QFT} is devoted to the construction of the weighted Gaussian measures and, finally, Section \ref{SEC:ENERGY} deals with the energy estimate.


\section
{Analytic and stochastic toolbox}
\label{SEC:TOOL}

\subsection{On the phase space}

Given $N \in \N$, we denote by  $\pi_N$ the frequency projector
on the (spatial) frequencies $\{|n| \leq N\}$: 
$$
(\pi_Nu)(x)=\sum_{|n|\leq N} \ft u_n \, e^{in\cdot x},  
$$

\noi
We then set
$$
\EE_N = \pi_N L^2(\T^3).
$$

\noi
Namely, $\EE_N$ is 
the finite-dimensional vector space of real-valued trigonometric polynomials of degree $\leq N$ endowed with the restriction of the $L^2(\T^3)$ scalar product.
The product space $\EE_N \times \EE_N$ is a finite dimensional real inner-product
space and thus there is a canonical Lebesgue measure on this space, which we
denote by $L_N$.
\noi
We also use $(\EE_N\times \EE_N)^\perp$
to denote
the orthogonal complement of $\EE_N\times \EE_N$ in $\H^\s(\T^3)$, $\sigma<s - \frac 12$.

%
%
%
%
%
%
%
%
%
%
%

\subsection{Besov spaces}

Let $B(\xi,r)$ denote the ball in $\R^3$ of radius $r>0$ centered at $\xi \in \mathbb{R}^3$ and let $\mathcal{A}$ denote the annulus $B(0,\frac 43)\setminus B(0, \frac{3}{8})$.
Letting $\N_0 = \N \cup\{0\}$, 
we define a sequence  $\{\chi_j\}_{j\in \N_0}$ by setting
\[
\chi_{0} = \wt\chi, \qquad 
\chi_j (\,\cdot\,)= \chi(2^{-j}\,\cdot\,), \qquad 
\text{and}\qquad  \sum_{j = 0}^\infty  \chi_j \equiv 1\]

\noi 
for some suitable  $\wt\chi, \chi \in C^\infty_c(\R^3; [0, 1])$ such that $\supp(\wt \chi) \subset B(0,\frac 43)$ 
and $\supp (\chi) \subset \mathcal{A}$.
We then define the Littlewood-Paley projector $\P_j$, $j \in \N_0$, 
by setting
\begin{equation*}
\P_j u(x) = \sum_{n \in \Z^3} \chi_j(n) \ft u(n) e^{ i n \cdot x}
\end{equation*}

\noi
for $u \in \mathcal{D}'(\T^3)$.

Given $s \in \R$ and $1\leq p, q \leq \infty$, 
the Besov space $B^s_{p,q}(\T^3)$ is the set of distributions $u \in \mathcal{D}'(\T^3)$ such that
\begin{equation}
\label{Besov}
\| u \|_{\B^s_{p,q}} 
= \Big\| \big\{2^{s j} \| \P_j u \|_{L^p_x}\big\}_{j \in \N_0} \Big\|_{\l_j^q} < \infty.
\end{equation}

\noi
We use the conventions $\vec{B}^s_{p,q} (\T^3)= B^s_{p,q}(\T^3) \times B^{s-1}_{p,q}(\T^3)$ 
and 
$\vec \C^s (\T^3) = \C^s (\T^3)\times \C^{s-1} (\T^3)$, 
where 
$\C^s (\T^3)= \B^s_{\infty,\infty}(\T^3)$ denotes the H\"older-Besov space.
Note that (i)~the parameter $s$ measures differentiability and $p$ measures integrability, 
(ii)~$H^s (\T^3) = \B^s_{2,2}(\T^3)$,
and (iii)~for $s > 0$ and not an integer, $\C^s(\T^3)$ coincides with the classical H\"older spaces;
see \cite{Graf}.

\begin{lemma}
The following estimates hold.

\noi
\textup{(i) (interpolation)} 
For  $0 < s_1  < s_2$, we have\footnote{We use the convention that the symbol $\lesssim$ indicates that inessential constants are suppressed in the inequality.}
\begin{equation}
\| u \|_{H^{s_1}} \les \| u \|_{H^{s_2}}^{\frac{s_1}{s_2}} \| u \|_{L^2}^{\frac{s_2-s_1}{s_2}}.
\label{interp}
\end{equation}

\noi
\textup{(ii) (immediate  embeddings)}
Let $s_1, s_2 \in \R$ and $p_1, p_2, q_1, q_2 \in [1,\infty]$.
Then, we have
\begin{align} 
\begin{split}
\| u \|_{\B^{s_1}_{p_1,q_1}} 
&\les \| u \|_{\B^{s_2}_{p_2, q_2}} 
\qquad \text{for $s_1 \leq s_2$, $p_1 \leq p_2$,  and $q_1 \geq q_2$},  \\
\| u \|_{\B^{s_1}_{p_1,q_1}} 
&\les \| u \|_{\B^{s_2}_{p_1, \infty}} 
\qquad \text{for $s_1 < s_2$},\\
\| u \|_{\B^0_{p_1, \infty}}
 &  \les  \| u \|_{L^{p_1}}
 \les \| u \|_{\B^0_{p_1, 1}}.
\end{split}
\label{embed}
\end{align}

\smallskip

\noi
\textup{(iii) (algebra property)}
Let $s>0$. Then, we have
\begin{equation}
\| uv \|_{\C^s} \les \| u \|_{\C^s} \| v \|_{\C^s}.
\label{alge}
\end{equation}

\smallskip

\smallskip

\noi
\textup{(iv) (Besov embedding)}
Let $1\leq p_2 \leq p_1 \leq \infty$, $q \in [1,\infty]$,  and  $s_2 = s_1 + 3\big(\frac{1}{p_2} - \frac{1}{p_1}\big)$. Then, we have
\begin{equation}
 \| u \|_{\B^{s_1}_{p_1,q}} \les \| u \|_{\B^{s_2}_{p_2,q}}.
\label{emb_b}
\end{equation}

\smallskip

\noi
\textup{(v) (duality)}
Let $s \in \mathbb{R}$
and  $p, p', q, q' \in [1,\infty]$ such that $\frac1p + \frac1{p'} = \frac1q + \frac1{q'} = 1$. Then, we have
\begin{equation}
\bigg| \int_{\T^3}  uv \bigg|
\le \| u \|_{B^{s}_{p,q}} \| v \|_{B^{-s}_{p',q'}},
\label{dual}
\end{equation}

\noi
where $\int_{\T^3} u v $ denotes  the duality pairing between $B^{s}_{p,q}(\T^3)$ and $B^{-s}_{p',q'}(\T^3)$.

\smallskip
	
\noi		
\textup{(vi) (fractional Leibniz rule)} 
Let $p, p_1, p_2, p_3, p_4 \in [1,\infty]$ such that 
$\frac1{p_1} + \frac1{p_2} 
= \frac1{p_3} + \frac1{p_4} = \frac 1p$. 
Then, for every $s>0$, we have
\begin{equation}
\| uv \|_{\B^{s}_{p,q}} \les  \| u \|_{\B^{s}_{p_1,q}}\| v \|_{L^{p_2}} + \| u \|_{L^{p_3}} \| v \|_{\B^s_{p_4,q}} .
\label{prod}
\end{equation}

\smallskip
	
\noi		
\textup{(vi) (product estimate)} 
Let $s_1 < 0 < s_2$ such that $s_1 + s_2 > 0$.
Then, we have
\begin{equation}
\| uv \|_{\C^{s_1}}\les  \| u \|_{\C^{s_1}}\| v \|_{\C^{s_2}}.
\label{prod2}
\end{equation}

\end{lemma}

\begin{proof}
While these estimates are standard, 
we briefly discuss their proofs for readers' convenience.
See also \cite{BCD} for details of the proofs in the non-periodic case.
The log convexity inequality \eqref{interp} 
and the duality \eqref{dual} follow from H\"older's inequality.
The first  estimate in~\eqref{embed} is immediate from the definition \eqref{Besov}, 
while the second one in \eqref{embed} follows from the $\l^{q_1}$-summability 
of $\big\{2^{(s_1 - s_2)j}\big\}_{j \in \N_0}$ for $s_1 < s_2$.
The last estimate in \eqref{embed} follows
from the boundedness of the Littlewood-Paley projector $\P_j$
and Minkowski's inequality.
The Besov embedding \eqref{emb_b} is a direct consequence
of Bernstein's inequality:
\[ \| \P_{j} u\|_{L^{p_1}} \les 2^{3j(\frac{1}{p_2} - \frac{1}{p_1})} \| \P_{j} u\|_{L^{p_2}}.\]

\noi
The algebra property~\eqref{alge}
is immediate from the following paraproduct decomposition
due to Bony \cite{Bony}:
\begin{align}
 uv = \sum_{j \in \N_0} \P_{j} u\cdot  S_j v
+ \sum_{j \in \N_0} \sum_{|j - k| \leq 1} \P_{j} u \cdot \P_{k} v
+ \sum_{k \in \N_0} S_k u \cdot \P_{k} v
\label{besov2}
\end{align}

\noi
with H\"older's inequality.
Here,  $S_j$ is given by 
\[ S_j u = \sum_{k \leq j - 2} \P_k u.\]

\noi
The fractional Leibniz rule  \eqref{prod} also follows from the paraproduct decomposition \eqref{besov2}.
In proving \eqref{prod} for the resonant product, i.e.~the second term
on the right-hand side of~\eqref{besov2}, 
one needs to proceed slightly more carefully:
\begin{align*}
\Bigg\| 2^{sm}
\bigg\| \P_m \Big(\sum_{j \in \N_0} \sum_{|j - k| \leq 1} \P_{j} u \cdot \P_{k} v\Big)\bigg\|_{L^p}\Bigg\|_{\l^q_m}
& \les \bigg\| \sum_{j \geq m - 10}
2^{s(m - j)} 2^{sj} \| \P_{j} u \|_{L^{p_1}} \| \P_{j} v\|_{L^{p_2}}\bigg\|_{\l^q_m}\\
& \les \| u \|_{\B^{s}_{p_1,q}}\| v \|_{L^{p_2}}, 
\end{align*}

\noi
where we used  Young's and H\"older's inequalities together with
the embedding: $L^{p_2}(\T^3) \hra B^0_{p_2, \infty}(\T^3)$ in the last step.
See also Lemma 2.84 in \cite{BCD}.
Lastly, the product estimate \eqref{prod2} follows from a similar consideration.
\end{proof}

\subsection{Wiener chaos estimate}

Let $\{ g_n \}_{n \in \N}$ be a sequence of independent standard Gaussian random variables defined on a probability space $(\O, \mathcal{F}, \mathbb{P})$, where $\mathcal{F}$ is the $\s$-algebra generated by this sequence. 
Given $k \in \N_0$, 
we define the homogeneous Wiener chaoses $\mathcal{H}_k$ 
to be the closure (under $L^2(\O)$) of the span of  Fourier-Hermite polynomials $\prod_{n = 1}^\infty H_{k_n} (g_n)$, 
where
$H_j$ is the Hermite polynomial of degree $j$ and $k = \sum_{n = 1}^\infty k_n$.\footnote{This implies
that $k_n = 0$ except for finitely many $n$'s.}
Then, we have the following Ito-Wiener decomposition:
\begin{equation*}
L^2(\Omega, \mathcal{F}, \mathbb{P}) = \bigoplus_{k = 0}^\infty \mathcal{H}_k.
\end{equation*}

\noi
See Theorem 1.1.1 in \cite{Nu}.
 We have the following classical Wiener chaos estimate.

\begin{lemma}\label{LEM:hyp}
Let $k \in \N_0$.
Then, we have 
 \begin{equation} \label{hyp}
 \Big( \E \big[|X|^p \big]\Big)^{\frac 1p} \leq (p-1)^\frac{k}{2} \Big( \E\big[|X|^2\big] \Big)^{\frac 12}
 \end{equation}

\noi
for any random variable $X \in \mathcal{H}_k$ and  any $2 \leq p < \infty$.

\end{lemma}

The estimate \eqref{hyp} is a direct corollary to the hypercontractivity of the Ornstein-Uhlenbeck semigroup due to Nelson \cite{Nelson2}
and the fact that any element $X \in \mathcal{H}_k$ is 
an eigenfunction for the Ornstein-Uhlenbeck operator
with eigenvalue $-k$. 

For our purpose, we need the following three facts: 
(i) If  $Z$ is a linear combination of $\{ g_n \}$, then $Z \in \mathcal{H}_1$.
 (ii) For $Z \in \mathcal{H}_1$, the random variable $Z^2 - \E [Z^2] \in \mathcal{H}_2$. 
 (iii) If $Y,Z \in \mathcal{H}_1$ are independent, then $YZ \in \mathcal{H}_2$.

The next lemma gives a regularity criterion  for stationary random distributions.
Recall that a random distribution $u$ on $\T^d$ is said to be stationary if $u(\,\cdot\,)$ and $u(x_0 + \cdot\,)$ have the same law 
for any $x_0 \in \T^d$. Moreover, we say that $u \in \mathcal{H}_k$ 
if 
$u(\varphi) \in \mathcal H^k$ for any test function $\varphi \in C^\infty(\T^d)$.

\begin{lemma}\label{LEM:MWX}
\textup{(i)} Let  $u$ be a stationary
random distribution on $\T^d$, belonging to 
$\mathcal{H}_{ k}$ for some $k \in \N_0$.
Suppose that there exists $s_0 \in \R$ such that 
\begin{align}
\E\big[|\ft u(n)|^2\big] \les \jb{n}^{-d - 2s_0}
\label{MWX1}
\end{align}

\noi
for any $n \in \Z^d$.
Then, for any $s < s_0$ and finite $p \geq 2$, 
we have $u  
\in L^p(\O; \C^s(\T^d))$.

\smallskip

\noi
\textup{(ii)} 
Let  $\{u_N\}_{N \in \N}$ be a sequence of stationary
random distributions on $\T^d$, belonging to $\mathcal{H}_{k}$ for some $k \in \N_0$.
Suppose that there exists $s_0 \in \R$
such that  
$u_N$ satisfies~\eqref{MWX1} for each $N \in \N$.
Moreover, suppose that 
there exists $\ta > 0$  such that 
\[\E\big[|\ft u_N(n) - \ft u_M(n)|^2\big] \les N^{-2\ta} \jb{n}^{-d - 2s_0}\] 

\noi
for any $n \in \Z^d$
and any $M \geq N \geq 1$. 
Then,  for any $s < s_0$ and finite $p \geq 2$, 
$u_N$  converges to some $u$  in $ L^p(\O; \C^s(\T^d))$.

\end{lemma}

The proof
 is a straightforward computation with  the Wiener chaos estimate (Lemma~\ref{LEM:hyp}). See \cite[Proposition 3.6]{MWX} for details
of the proof of Part~(i). Part (ii) follows from similar considerations.

\subsection{Truncated NLW dynamics:~well-posedness and approximation}
\label{SEC:NLW}

In the following, we often work at the level of the truncated dynamics in order to rigorously justify calculations. As such, in this subsection, we briefly go over the well-posedness theory and approximation results
of the following Cauchy problem for the truncated NLW on $\T^3$:
\begin{equation}\label{NLW-sys_CP}
\begin{cases}
\partial_t u=v\\
 \partial_t v=\Delta u-\pi_N\big((\pi_N u)^3\big)\\
 ( u, v)|_{t = 0} = (u_0, v_0), 
 \end{cases}
\end{equation}

\noi
where $N\geq 1$ and $\pi_N$ denotes the projector onto spatial frequencies  $\{|n| \leq N\}$. 
We also use the following shorthand notations: 
\begin{align*}
u_N = \pi_N u \qquad \text{and} \qquad v_N = \pi_N v.
\end{align*}

\noi
 We allow $N=\infty$ with the convention $\pi_{\infty}={\rm Id}$, which reduces \eqref{NLW-sys_CP} to \eqref{NLW}. 

For the (untruncated) NLW \eqref{NLW}, the conserved energy is given by 
\begin{equation*}
E(\vec{u}) = \frac 12 \int_{\T^3} \big(|\nabla u|^2 + v^2\big) + \frac 14 \int_{\T^3} u^4.	
\end{equation*}

\noi
The truncated system \eqref{NLW-sys_CP} also has the following conserved energy:
\begin{align}
E_N(\vec{u}) \label{E2}
& =\frac{1}{2}\int_{\T^3}\big(|\nabla u|^2+v^2\big) +\frac{1}{4}\int_{\T^3}(\pi_N u)^4.
\end{align}

In the following two lemmas, 
we state the classical well-posedness theory for \eqref{NLW-sys_CP}
and the relevant dynamical properties.

\begin{lemma} \label{LEM:CP_GWP}
	Let $\s \geq 1$ and $N \in \N \,\cup  \{\infty\}$.
	Then, the truncated  NLW \eqref{NLW-sys_CP} is globally well-posed in $\H^\s(\T^3)$.
	Namely, given any  $(u_0,v_0)\in \H^\s(\T^3)$, 
	there exists a unique global solution to \eqref{NLW-sys_CP} in $C(\R; {\H}^\s(\T^3))$, 
	where the dependence on initial data is continuous. 
	Moreover, if we denote by $\Phi_N(t)$ the  data-to-solution map at time $t$, then $\Phi_N(t)$ is a continuous bijection on ${\H}^\s(\T^3)$ for every $t\in\R$, satisfying the semigroup property:
	$$
	\Phi_N(t+\tau)=\Phi_N(t)\circ \Phi_N(\tau) 
	$$
	
	\noi
	for any $t, \tau \in \R$.
\end{lemma}

The global well-posedness result stated in Lemma \ref{LEM:CP_GWP} follows
from a standard local well-posedness theory
along with the conservation of the truncated energy $E_N(\vec{u})$.
See \cite[Lemma 2.1]{OTz3} for the proof in the two-dimensional case.\footnote{This is in the context of the nonlinear Klein-Gordon equation but the proof can be easily adapted.}
The same proof applies to the three-dimensional case
in view of the Sobolev embedding $H^1(\T^3)\subset L^6(\T^3)$
(with a small modification at the zeroth frequency).

\begin{lemma} \label{LEM:CP_APPRX}
	\textup{(i) (Growth bound)} Given $\s \geq 1$, 
	we denote by $B_R$  the ball of radius $R>0$ in $\H^\s(\T^3)$
centered at the origin. 
Then, for any given $T > 0$,
	there exists $C(R, T) > 0$ such that 
	\begin{align}
	\Phi_N(t) (B_{R}) \subset B_{C(R, T)}
	\label{growth1}
	\end{align}
	
	\noi
	for any $t\in [0, T]$ and $N \in \N \cup \{\infty\}$.

	\smallskip
	
	\noi
	\textup{(ii) (Approximation)}
	Let $\sigma\geq 1$, $T>0$,  and $K$ be a compact set in $\H^\s(\T^3)$. 
	Then,  for every $\eps>0$,
	there exists  $N_0\in \N$ such that 
	\begin{equation*}
	\|\Phi(t)(\vec{u})-\Phi_N(t)(\vec{u})\|_{\H^\sigma(\T^3)}<\eps
	\end{equation*}

	\noi
	for any $t \in [0, T]$, $\vec{u}\in K$,  and  $N\geq N_0$. Hence, 
we have
	\begin{equation*}
	\Phi(t)(K)\subset \Phi_N(t)(K+B_\eps).
	\end{equation*}
	
	\noi
	for  any $t \in [0, T]$ and
 $N\geq N_0$.
	Here, $\Phi(t)$ denotes the solution map
	$\Phi_{\infty}(t)= \Phi_\NLW(t)$
	for the \textup{(}untruncated\textup{)} NLW \eqref{NLW}.
\end{lemma}

\begin{proof}
The solution $\vec{u}= (u, v)$ to \eqref{NLW-sys_CP} satisfies
the following Duhamel formulation:
\begin{align}
\begin{split}
u(t) & = S(t) (u_0, v_0) - \int_0^t \frac{\sin((t - t')|\nb|)}{|\nb|} 
\pi_N\big((\pi_N u)^3\big)(t') dt', \\
v(t) & = \dt S(t) (u_0, v_0) - \int_0^t \cos((t - t')|\nb|)
\pi_N\big((\pi_N u)^3\big)(t') dt', 
\end{split}
\label{D1}
\end{align}

\noi
where $S(t)$
denotes the linear wave propagator given by 
\[S(t) (u_0, v_0) = \cos (t |\nb|)u_0 +  \frac{\sin (t |\nb|)}{|\nb|}v_0.\]

\noi
From the fractional Leibniz rule \eqref{prod} and \eqref{emb_b}, we have
\begin{align}
\begin{split}
\| u^3 \|_{H^{\s-1}} & \les \|  u \|_{B^{\s-1}_{6, 2}}\|u\|_{L^6}^2  
 \les \|  u \|_{H^{\s}}\|u\|_{H^1}^2
\end{split}
\label{D2}
\end{align}

\noi
for $\s \geq 1$.
Then, from \eqref{D1} and \eqref{D2}
with 
the conservation of the truncated energy $E_N$ in \eqref{E2}, 
we have\footnote{The factor $1+ |t|$ appears
 in controlling 
 the zeroth frequency: $\frac{\sin( (t - t') |\nb|)}{|\nb|} = t - t'$.}
\begin{align*}
\|\vec{u}(t) \|_{\H^\s}
& \leq \| (u_0, v_0)\|_{\H^\s}
+ C (1+ |t|) \int_0^t \|  u(t') \|_{H^{\s}}\|u(t') \|_{H^1}^2 dt'\\
& \leq \| (u_0, v_0)\|_{\H^\s}
+ C (1+ |t|)  \cdot E_N(u_0, v_0) \int_0^t \|  (u, v)(t') \|_{\H^{\s}} dt'.
\end{align*}

\noi
Hence, the growth bound \eqref{growth1} follows from 
Gronwall's inequality.

The approximation property (ii) follows from a modification of the local well-posedness argument.
Since the argument is standard, we omit details.
See, for example, our previous works:
Proposition 2.7 in \cite{Tzvet}
and Lemma 6.20/B.2 in \cite{OTz}.
\end{proof}


\section{Proof of Theorem \ref{THM:NLW}}
\label{SEC:PLAN}

In this section, we present the proof of Theorem \ref{THM:NLW}.
We first present a general framework of the strategy.
We then introduce a renormalized energy
and discuss further refinements required for our problem.
In Subsection \ref{SUBSEC:proof2}, we prove Theorem \ref{THM:NLW} by assuming
the construction of the weighted Gaussian measure
(Proposition \ref{PROP:QFT0})
and 
the renormalized energy estimate (Proposition~\ref{PROP:Nenergy}).
We present the proofs of Propositions~\ref{PROP:QFT0}
and~\ref{PROP:Nenergy} in Sections \ref{SEC:QFT} and \ref{SEC:ENERGY}.

\subsection{Strategy of the proof}

In \cite{Tzvet},
the third author introduced a general 
strategy, combining  PDE techniques and stochastic analysis to prove quasi-invariance of Gaussian measures
under nonlinear Hamiltonian PDE dynamics.
In the following, we briefly describe the key ideas behind this method \cite{Tzvet, OTz3},
using NLW on $\T^d$ as an example.
See also \cite{OTz2} for a survey on this subject.
Note that we keep our discussion at a formal level
and that some steps need to be justified
by working at the level of the truncated dynamics \eqref{NLW-sys_CP}.

Let $\Phi = \Phi_{\NLW}$ as in the previous section.
 In order to prove quasi-invariance of $\muu_s$ under $\Phi$,  
we would like to show $\muu_s(\Phi(t)(A)) = 0$
for any $t \in \R$ and any  measurable set $A \subset \H^\s(\T^d)$ with  $\muu_s(A) = 0$. 
Here, 
 $\s < s + 1 - \frac d2$ denotes 
 the regularity of samples on $\T^d$ under $\muu_s$. 
The main idea is to study the evolution of 
\begin{align*}
\muu_s(\Phi(t)(A))
=   Z_{s}^{-1}\int_{\Phi(t)(A)}
e^{-\frac 12 \|\vec u\|_{\H^{s+1}}^2} d\vec u
\end{align*}

\noi
for a general measurable set $A \subset \H^\s(\T^d)$
and to control  the growth of  $\muu_s(\Phi(t)(A))$ in time. 

The main goal is  show a differential inequality of the form:
\begin{align} \label{measure_diffeqn}
\frac{d}{dt} \muu_s(\Phi(t)(A)) \leq C p \big\{ \muu_s(\Phi(t)(A)) \big\}^{1 - \frac 1p}	
\end{align}

\noi
for all sufficiently large but finite $p$. Using \eqref{measure_diffeqn}, one can show that $\frac{d}{dt} \muu_s(\Phi(t)A)^\frac 1p$ is bounded, which yields quasi-invariance for short times after choosing $p$ appropriately. See \cite[Proposition 5.3]{OTz3} for details\footnote{This is a refinement of \cite[Lemma 7.3]{Tzvet}, based on an argument due to Yudovich \cite{Y}, which handles the simpler case, where $p$ in \eqref{measure_diffeqn} is replaced by $p^\beta$
for some  $0 \leq \beta < 1$.}. In this argument, the linear power of $p$ in the prefactor
 of the right-hand side of \eqref{measure_diffeqn} is crucial.

By applying a change-of-variable formula
(see Lemma \ref{LEM:cov} below), 
 we have
\begin{align}
\muu_s(\Phi(t)(A)) 
 \text{``}=\text{''}
Z_s^{-1}\int_{A} 
e^{-\frac 12  \|\Phi(t)(\vec u)\|_{\H^{s+1}}^2} d\vec u.
\label{A1}
\end{align}

\noi
For the truncated dynamics \eqref{NLW-sys_CP}, 
the formula \eqref{A1} can be justified
via invariance of the Lebesgue measure and bijectivity of the flow $\Phi_N$.
See Lemma \ref{LEM:cov} below.
Fix $t_0 \in \R$. 
Then, by taking a time derivative, we arrive at 
\begin{align}
\begin{split}
\frac{d}{dt} \muu_s(\Phi(t)(A)) \bigg|_{t = t_0}
& = -\frac12Z_s^{-1}\int_{\Phi(t_0)(A)} 
\frac{d}{dt} \bigg( \|\Phi(t)(\vec u)\|_{\H^{s+1}}^2\bigg)
e^{-\frac 12  \|\Phi(t)(\vec u)\|_{\H^{s+1}}^2} d\vec u\bigg|_{t = 0} \\
& =
-\frac 12 \int_{\Phi(t_0)(A)} 
\frac{d}{dt} \bigg( \|\Phi(t)(\vec u)\|_{\H^{s+1}}^2\bigg)\bigg|_{t = 0}
d\muu_s . 
\end{split}
\label{A1a}
\end{align}

\noi
This reduction of the analysis to 
 that at $t = 0$, exploiting the 
group property $\Phi(t_0 + t) = \Phi(t) \Phi(t_0)$
was inspired from  the work \cite{TzV1}.

Suppose that we had an effective energy estimate  (with smoothing) of the form:
\begin{align}
\frac{d}{dt} \|\Phi(t) (\vec u)\|_{\H^{s+1}}^2\bigg|_{t = 0}
\text{``}\leq\text{''} \, C(\| \vec u  \|_{\H^1})
\| \vec u \|_{\vec \C^{\s}}^{\theta}
\label{A2}
\end{align}

\noi
for some $\theta \leq 2$.
Then, the desired estimate
\eqref{measure_diffeqn} would follow
from \eqref{A1a} and \eqref{A2} 
(with 
an additional  cutoff on the conserved energy $E(\vec{u})$ in \eqref{A1}). 
More precisely, 
we obtain~\eqref{measure_diffeqn}
by 
inserting \eqref{A2} into \eqref{A1a}, 
 applying H\"older's inequality with respect to $\muu_s$, and 
 then using the Wiener chaos estimate (Lemma \ref{LEM:hyp}) to obtain (sub-)linear $p$ dependence on the $L^p(\muu_s)$-norm of $\| \vec u \|_{\vec \C^{\s}}^{\theta}$, 
 while 
we   use a cutoff on the conserved energy $E(\vec{u})$ to control $C(\|\vec{u}\|_{\H^1})$ in~\eqref{A2}.
   See the proof of Proposition \ref{PROP:meas2} below for the full detail.
 We emphasize that,  due to the $p$-dependence of the constant in the Wiener chaos estimate 
 (Lemma \ref{LEM:hyp}), we can only afford to place two factors of 
$\vec u$ in the stronger H\"older-Besov  $\vec \C^{\s}$-norm in the energy estimate~\eqref{A2}; a higher order Wiener chaos would cause superlinear $p$-dependence in \eqref{measure_diffeqn}. All the other factors are then placed in the (weaker) $\H^1$-norm, 
which is controlled by the conserved energy $E(\vec u)$ in \eqref{E2}.

In  \cite{Tzvet}, 
the third author established 
an energy estimate of the form \eqref{A2} for the BBM equation
by consideration  
in the spirit of  quasilinear hyperbolic PDEs
(namely, integration by parts in $x$).
Unfortunately, an energy estimate of the form \eqref{A2} 
does not hold in general for nonlinear Hamiltonian PDEs.
In \cite{OTz, OTz3}, 
the second and third authors 
circumvented this problem by 
 introducing a modified energy:
\begin{align*}
E_s(\vec u) = \frac 12 \|\vec u\|_{\H^{s+1}}^2 + R_s(\vec u)
\end{align*}

\noi
with a suitable correction term $R_s(\vec u)$ 
such that the desired energy estimate of the form~\eqref{A2}
holds for this modified energy.
 By following the strategy described above,
they first established
  quasi-invariance of the weighted Gaussian measure
associated with this modified energy:
\[d\rhoo_s = Z_s^{-1} e^{-E_s(\vec u)} d\vec u = Z_s^{-1} e^{-R_s(\vec{u})} d\muu_s \]

\noi
(with a cutoff on a conserved quantity).
Then, 
 quasi-invariance of $\muu_s$ followed from 
the mutual absolute continuity of  $\muu_s$ and $\rhoo_s$.

For  Schr\"odinger-type equations, 
 modified energies were introduced
 by the normal form method (namely, integration by parts in time);
 see \cite{OTz, OST, FT}.
In \cite{OTz3}, 
the second and third authors derived a modified
energy for NLW on $\T^2$
based on integration by parts in $x$
but a certain renormalization was needed to control singularity.
We will 
describe  the details
of this derivation 
in the next subsection.

\smallskip

\noi
{\bf Summary:}
The study of quasi-invariance has therefore been reduced to two steps: 
(i)~the construction of the weighted Gaussian measure $\rhoo_s$
and  (ii) establishing
an effective energy  estimate on $\partial_t E_s(\vec{u})\big|_{t=0}$.

\subsection{Renormalized energy for  NLW}
\label{SUBSEC:cor}
In this subsection, 
we present a discussion on a modified energy
for our problem.
See \eqref{Es} below for the full modified energy.
In the following, 
we fix $ \s = s +1 - \frac d2 - \eps \geq  1 $ for some small $\eps > 0$
and let $B_R$ denotes the ball of radius $R >0$  in $\H^\s(\T^d)$ centered at the origin. 
Fix a frequency  cutoff size $N$ and, 
instead of using (a suitable truncated version of) 
the energy of $\muu_s$, let us consider  the following natural energy to work with for the wave equation (see Remark \ref{REM:zero_freq}):
\[ \frac 12 \int_{\T^d} (D^s v_N)^2 + \frac 12 \int_{\T^d} (D^{s+1}u_N)^2, \]

\noi
where $D^s = (-\Delta)^{\frac s2}$ denotes the Riesz potential of order $s$. 
Fix an even integer $s \geq 4$ and 
let  $\vec{u}  = (u, v)$ be a solution to the truncated NLW~\eqref{NLW-sys_CP}.
Then,   the Leibniz rule yields
\begin{align}
\begin{split}
 \dt \bigg[\frac 12  \int_{\T^d} (D^s v_N)^2&  + \frac 12 \int_{\T^d} (D^{s+1}  u_N)^2
\bigg]  
  =  \int_{\T^d} (D^{2s} v_N)(-u_N^3) \\
& =
 -3\int_{\T^d} D^sv_N D^s u_N\, u_N^2  \\
& \hphantom{X}
+ \sum_{\substack{ |\al|+|\be|+|\g| = s\\
 |\al|,|\be|,|\g|<s
}}
c_{\al,\be,\g}
\int_{\T^d}
D^sv_N\cdot\dd^\al u_N\cdot \dd^\be u_N\cdot \dd^\g u_N
\end{split}
\label{H1b}
\end{align}

\noi
for some combinatorial constants $c_{\al,\b,\g}$ that depend only on $s$,
 where 
$\dd^\al$ denotes $\dd_{x_1}^{\al_1} \cdots \dd_{x_d}^{\al_d}$
for a multi-index $\al = (\al_1, \dots, \al_d)$.
%
Samples $\vec{u}$ under the Gaussian measure $\muu_s$
belong almost surely to 
$\vec \C^\s(\T^d)\setminus \vec \C^{s+1 - \frac d2}(\T^d)$ for $\s < s +1 - \frac {d}{2}$. 
The main issue is how to treat
 $D^s v_N$ on the right-hand side of \eqref{H1b} due to its low regularity $\s - 1$. 
It turns out that all but the first term on the right-hand side of \eqref{H1b} 
can be treated by integration by parts.  See Remark~\ref{REM:IBP}.
 As for the first term, recalling from \eqref{NLW-sys_CP} that $v_N = \dt u_N$, 
 we have 
\begin{align}
-3\int_{\T^d} D^sv_N  D^s u_N\,  u_N^2
 = - \frac 32 \dt \bigg[\int_{\T^d} (D^su_N)^2u_N^2 \bigg]
+ 3 \int_{\T^d} (D^su_N)^2\, v_N u_N .
\label{H1}
\end{align}

\noi
The terms on the right-hand side of \eqref{H1} are better behaved than 
that on the left-hand side since $D^s$ no longer falls on the less regular term $v$.
This motivates us to define a modified energy with  a correction term of the form:
\begin{equation*}
R_s(\vec{u}) = 
\frac{3}{2} \int_{\T^d} (D^s u_N)^2 u_N^2.
\end{equation*}

When $d=1$, 
this choice of the correction term allows
us to define a suitable modified energy
and to construct the weighted Gaussian measure 
associated with this modified energy
(modulo an issue at the zeroth frequency).
When $d = 2$ or 3, however, 
we have $u \notin \C^s(\T^d)$ almost surely
and thus the limiting expression
$(D^s u)^2$  
is ill defined since it is the square of a  distribution of negative regularity.
 Moreover, the singular term $(D^s u)^2$ appears in both terms 
on the right-hand side of \eqref{H1}.
As such, we have issues at the level of both the energy and its time derivative,  which propagate to both the construction of the weighted Gaussian measure and  the energy estimate.

Motivated by Euclidean quantum field theory, we introduce a renormalization. 
This amounts to
replacing $(D^s u)^2$ by $(D^s u)^2 - \infty$, suitably interpreted;
given $N \in \N$, we replace $(D^s u_N)^2 $  in \eqref{H1} by 
$\DD$, where 
\begin{align}
Q_{s, N}(f) \stackrel{\text{def}}{=} (D^sf)^2 - \s_N
\label{H1c}
\end{align}

\noi
and $\s_N$ is given by 
\begin{align}
\s_N \stackrel{\text{def}}{=}
\E_{\muu_s}
\Big[ (D^s \pi_N u)^2\Big] 
\sim 
 \sum_{\substack{n \in \Z^d\\1\leq |n|\leq N}} 
\frac{1}{|n|^2}
\sim \begin{cases}
\log N & \text{for } d = 2, \\
N & \text{for } d = 3, 
\end{cases}
\label{sigma0}
\end{align}

\noi
as $N \to \infty$. 
The crucial observation in \cite{OTz3}
is that 
 the effect of the renormalization for the two terms on the right-hand side 
 in \eqref{H1} 
precisely  cancels each other, since
\begin{align*}
- \frac 32 \s_N \dt \bigg[\int_{\T^d} u_N^2 \bigg]
+ 3 \s_N \int_{\T^d} v_N u_N  = 0,
\end{align*}

\noi
where we used the equation \eqref{NLW-sys_CP}.
As a result,  we obtain
\begin{align}
-3\int_{\T^d} D^sv_N & D^s u_N\,  u_N^2
= - \frac 32 \dt \bigg[\int_{\T^d}  \DD u_N^2 \bigg]
+ 3 \int_{\T^d} \DD   v_N u_N .
\label{H2}
\end{align}

\noi
In view of \eqref{H1b} and \eqref{H2}, we
define the renormalized energy $\HH_{s, N}(\vec{u})$ by 
\begin{equation}
\HH_{s, N}(\vec{u})  = 
 \frac 12 \int_{\T^d} (D^{s+1}  u)^2
+  \frac 12  \int_{\T^d} (D^s v)^2 
 + \frac 32  \int_{\T^d} \DD u_N^2 .
\label{H3}
\end{equation}

\noi
Then, we have 
\begin{align}
\begin{split}
\dt \HH_{s, N}(\vec{u})  
& = 3 \int_{\T^d} \DD   v_N u_N \\
& \hphantom{X} + \sum_{\substack{ |\al|+|\be|+|\g| = s\\
 |\al|,|\be|,|\g|<s
}}
c_{\al,\be,\g}
\int_{\T^d}
D^sv_N\cdot\dd^\al u_N\cdot \dd^\be u_N\cdot \dd^\g u_N.
\end{split}
\label{H3a}
\end{align}

\noi
Note that we have renormalized both the energy and its time derivative at the same time.
The considerations above motivate the definition of the renormalized weighted Gaussian measure:
\begin{equation}
d\rhooo_{s,r,N} = Z_{s,N,r}^{-1}\textbf{1}_{\{E_N(\vec{u}) \leq r\}} e^{-\HH_{s, N}(\vec{u})} d\vec{u}, 
\label{cutoff_rho}
\end{equation}

\noi
where $E_N(\vec u)$ is as in \eqref{E2}.
The energy cutoff in \eqref{cutoff_rho} is necessary to construct this measure due to an issue with the zeroth frequency  (see Remark \ref{REM:zero_freq}).

\begin{remark}\rm

If $\vec u$ is distributed according to 
the Gaussian measure $\muu_s$,
then 
we can apply  Wick renormalization to $(D^s u_N)^2$
and obtain the  Wick power $ :\!(D^s u_N)^2\!:$.
Here, Wick renormalization corresponds 
 the orthogonal projection onto 
a (second) homogeneous Wiener chaos under $L^2(\muu_s)$.
In this case, we have 
\[ :\!(D^s u_N)^2\!: \,   =  \DD.\]

\noi
This renormalization allows us to take a limit
$:\!(D^s u)^2\!: \, = \lim_{N \to \infty} :\!(D^s u_N)^2\!:$ in a suitable space
(see Lemmas \ref{LEM:QFT2} and \ref{LEM:Y} below).
In the discussion above for deriving the renormalized energy $\HH_{s, N}$, however, 
$\vec u$  denotes a solution
to \eqref{NLW-sys_CP}
and  a notation such as $:\!(D^s u_N)^2\!:$ is not well defined.
This is the reason  we needed to introduce
$Q_{s, N}$ in \eqref{H1c}.

\end{remark}

\begin{remark}\rm

This simultaneous renormalization of the energy and its time derivative 
does not introduce any modification to the original truncated equation \eqref{NLW-sys_CP}
since its Hamiltonian $E_N(\vec u)$ remains unchanged.
We also point out two (related) interesting observations: 
(i)~renormalization is usually applied in the handling of rough functions, whereas we use renormalization in
the context of high regularity solutions,
 and (ii)~the simultaneous renormalization is introduced only as a tool to prove Theorem \ref{THM:NLW}. 

\end{remark}

\begin{remark} \label{REM:IBP}\rm
	In view of the regularity of $\vec{u}$ under $\muu_s$, it may seem that some of the lower order terms under the sum on the right-hand side of \eqref{H3a} are divergent as $N \to \infty$: for example, when $|\al| = s-1$, $|\be| = 1$, and $\g = 0$.
	However, by integration by parts (in $x$) and the independence of $u$
	and $v$, they turn out to be convergent without any renormalization.
	See Propositions \ref{PROP:random_distr} and \ref{PROP:energyestimate}.
\end{remark}

\medskip

\noi
$\bullet$ {\bf Problem (i): Construction of the weighted Gaussian measure.} The problem of constructing the limiting weighted Gaussian measure
measure $\rhooo_{s,r} = \lim_{N \rightarrow \infty} \rhooo_{s,r,N}$ bears some similarity with the problem of constructing the 
$\Phi^4$-measures. 
First of all, the need for renormalization in~\eqref{H3}  means that the positivity of the random variable 
$\int (D^s u)^2 u^2$ is destroyed. 
Moreover,  there is a similarity between the measures themselves;
despite not having the simple algebraic structure of the $\Phi^4$-measure, the term $\int (D^s u)^2 u^2$ is quartic in $u$.

In \cite[Proposition 3.1]{OTz3}, the second and third authors exploited these similarities 
and modified Nelson's construction of  the $\Phi^4_2$-measure
to construct the desired weighted Gaussian measure $\rhooo_{s, r}$ in the two-dimensional case. We summarize the argument for the reader's convenience; let 
$X_N = \int_{\T^2} :\!(D^s u_N)^2\!: u_N^2  $ and $X$ be its limit as $N \rightarrow \infty$. Using the energy cutoff, we have $X_N \geq -C_r \log N$, which is the scale at which the semi-boundedness of $X_N$ blows up in the frequency cutoff parameter $N$. For any $K > 0$, we have the following decomposition.
\begin{equation}
\mu_s \left( e^{-X} > e^K \right) \leq \mu_s \left( e^{X_N-X} > e^{K - C_r\log N} \right).
\end{equation}
One then chooses the frequency cutoff parameter $N$ so that 
the rate $C_r \log N$ of divergence of $X_N$ matches the tail integrability parameter $K$. This gives a good tail bound on $e^{-X}$ and establishes its integrability. Note that this argument heavily uses the logarithmic divergence rate \eqref{sigma0}
of the renormalization constants 
when $d = 2$ (although one can show that it still holds for divergences of rate $N^\eps$ for  small $\eps > 0$). As a consequence, it breaks down in the three-dimensional case due to the stronger algebraic divergence rate \eqref{sigma0}
of the renormalization constants $\s_N$; see Remark 3.6 in~\cite{OTz3}. 

In order to construct $\rhooo_{s,r}$, we use the techniques introduced 
in a recent paper~\cite{BG} by Barashkov and Gubinelli, where the partition functions of the $\Phi^4_{2}$- and $\Phi^4_3$-measures were analyzed by way of variational formulas. 
We, however, point out that the construction of $\rhooo_{s,r}$ is much easier than the construction of 
the $\Phi^4_3$-measure. The fundamental reason for this difference 
is that in the term 
$\int_{\T^3} (D^su)^2 u^2 $,  one takes the second power of the irregular distribution $D^s u$, rather than the fourth power. Indeed, in the case of the $\Phi^4_3$-measure,  
renormalization beyond Wick ordering is required, resulting in the measure being singular with respect to its underlying the Gaussian measure \cite{BG2}. By contrast, we show that the measures $\rhooo_{s,r}$ require only Wick ordering and are in fact still absolutely continuous with respect to the underlying Gaussian measure.\footnote{In order to avoid an issue at the zeroth frequency, 
we  need to make a modification to the renormalized energy 
$\HH_{s, N}(\vec{u})$.
This  leads to a slightly different weighted Gaussian measure.  See~\eqref{Es}, \eqref{K1}, and~\eqref{K0} below.} Moreover, we are able to construct $\rhooo_{s,r}$, 
 using a simpler version of the variational approach used in \cite{BG},  more precisely, based on \cite[Lemma 1]{BG} rather than \cite[Theorem 2]{BG};
 see Remark \ref{REM:BD}. It may be possible to construct $\rhooo_{s,r}$ using more classical techniques developed in the subtler context of constructing the $\Phi^4_3$-measure, such as phase cell expansions \cite{GJ} or renormalization group methods \cite{GA}. We choose to use the variational approach because it leads to a relatively simple argument and, moreover, has more common mathematical ground with the analysis of wave equations and other dispersive PDEs (i.e.~making heavy use of harmonic analysis). 
 We also mention recent works
 \cite{ORSW, OOT1, Bring, OSTolo, Rob}, 
where the  variational approach was used
in the construction of invariant measures for 
 dispersive PDEs.

One technical issue with the construction of $\rhooo_{s,r}$ is that it is not clear whether the term $\int (D^s u)^2 u^2$ is good enough to control the large-scale behavior (= low frequency part) of $u$.
In the following, we circumvent this problem by introducing a new renormalized energy $E_{s, N}(\vec u)$ in \eqref{Es}
by adding  the energy $E_N(\vec u)$ in \eqref{E2} (plus an extra term
controlling the zeroth Fourier coefficient of $u$) to the renormalized energy $\HH_{s,N}(\vec{u})$ in \eqref{H3}.
This allows us to use the potential energy term~$\frac 14 \int u_N^4$
in \eqref{E2} to  get rid of the need of the energy cutoff 
$\textbf{1}_{\{E_N(\vec{u}) \leq r\}}$.
The effect is to change the underlying Gaussian measure $\muu_s$ to a different Gaussian measure  $\nuu_s$, which will be shown to be equivalent to $\muu_s$ by Kakutani's theorem. 
See Lemma \ref{LEM:equiv} below.

\medskip

\noi
$\bullet$ {\bf Problem (ii): Energy estimate.}
In the two-dimensional case \cite{OTz3}, 
it was not possible to establish an energy estimate 
of the form \eqref{A2}.
Instead, it was shown that 
\begin{align}
\Big| \dt E_{s, N}(\pi_N \Phi_N(t)(\vec{u}))|_{t = 0} \Big|
\les C(\|\vec{u}\|_{\H^1}) F(\vec{u}).
\label{Nest1}
\end{align}

\noi
for a suitable renormalized energy.
Here, $F(\vec{u})$ denotes complicated expressions that contain high regularity information on $\vec{u}$ 
such as the $\vec{W}^{\s,\infty}$-norm as well as 
the renormalized second power $\int_{\T^2} \DD $.
As mentioned above, all but two factors need to be placed in the weaker $H^1$-norm so that $F(\vec{u})$ is at most quadratic in $\vec{u}$, which implies that $F(\vec{u}) \in \mathcal{H}_2$.
 This allows us to obtain  the right growth bound 
 of the form \eqref{measure_diffeqn}
 after applying the Wiener chaos estimate (Lemma \ref{LEM:hyp}).
Here, it is crucial to study the energy 
estimate \eqref{Nest1}
at time $t=0$ to exploit the Gaussian initial data in in \eqref{series}.
 In \cite{OTz3}, 
the energy estimate~\eqref{Nest1} involved a delicate quadrilinear Littlewood-Paley expansion
balancing  the interplay between the energy conservation and the higher order regularity. 
As pointed out in \cite{OTz3}, the estimate of the form \eqref{Nest1} fails for the three-dimensional case.
 
In a recent paper \cite{PTV}, 
Planchon, Visciglia, and the third author proved 
quasi-invariance of the Gaussian measures
under the dynamics of the (super-)quintic nonlinear Schr\"odinger equations (NLS)
 on $\T$
  by establishing a novel 
  energy estimate.
   The idea is to exploit
a deterministic growth bound  \eqref{growth1}  on solutions. 
Then, the required energy estimate takes 
the following form:\footnote{In the case of NLS, we have $u$ instead of $\vec{u} = (u, v)$.
For the sake of presentation, we keep the notation adapted to the NLW  context.}
\begin{align}
\Big| \dt E_{s, N}(\pi_N \Phi_N(t)(\vec{u}) )\Big|
\le C \big(1 + \|\Phi_N(t) (\vec{u}) \|_{\H^\s}^k\big).
\label{Nest4}
\end{align}

\noi
Here, $k > 0$ can be {\it any} positive number.
The main point is that if we start dynamics with a measurable set $A \subset B_R $,
then  \eqref{Nest4} with the growth bound \eqref{growth1} yields
\begin{align*}
\Big|\ind_{A}(\vec{u}) \cdot  \dt E_{s, N}(\pi_N \Phi_N(t)(\vec{u})) \Big|
\leq C\Big|\ind_{B_{C(R, T)}}(\vec{u})\cdot 
 \big(1 + \|\vec{u} \|_{\H^\s}^k\big) \Big|
\le C(R)^k
\end{align*}

\noi
for any $t \in [0, T]$ and $N \in \N \cup\{\infty\}$.
This control allows us to  prove quasi-invariance
for each measurable set $A \subset B_R$
(in the sense of \eqref{N2} below).
Then, by a soft argument, we can conclude quasi-invariance
of the Gaussian measure $\muu_s$. The main advantage of this argument 
is that we are allowed to place any power $k$ 
in the stronger $\H^\s$-norm. 
Note that the energy estimate \eqref{Nest4} is entirely deterministic
and hence there is no need to reduce the analysis to time $t= 0$.

In this paper, we combine these two approaches described above 
and establish 
an energy estimate of the form:
\begin{align}
\Big| 
\ind_{B_{R}}(\vec{u})\cdot 
\dt E_{s, N}(\pi_N \Phi_N(t)(\vec{u}))|_{t = 0} \Big|
\le \ind_{B_R}(\vec{u}) C(\|\vec{u}\|_{\H^\s}) F(\vec{u}) 
\le
C(R)F(\vec{u}). 
\label{Nest7}
\end{align}

\noi
In the actual application of this energy estimate, 
in place of $\ind_{B_R}(\vec u)$ in~\eqref{Nest7}, 
we have $\ind_{\Phi_N(t_0)(B_R)}(\vec u)$ for some $0 < t_0\le T$, 
which can be majorized by 
$\ind_{B_C(R, T)}(\vec u)$
thanks to the deterministic growth bound~\eqref{growth1};
see  \eqref{eq growth}.
As for  $F(\vec u)$ 
in~\eqref{Nest7}, 
we  use the Wiener chaos estimate (Lemma~\ref{LEM:hyp}). 
The fact that we have access to the stronger $\H^\s$-norm (rather than $\H^1$-norm as in \eqref{Nest1})
allows us to get by with a softer energy estimate.
Moreover, 
in our case, $F(\vec{u})$ is given in an  explicit manner (see Proposition~\ref{PROP:energyestimate}). 
It contains products of derivatives of $u_N$ and $v_N$
 as well as the $\C^{-1-\eps}$-norm of the Wick power
$\DD = (D^s u_N)^2 -\s_N$.
By proceeding as in~\cite{MWX}, 
we establish regularity properties
of these random distributions
in Proposition \ref{PROP:random_distr}.
These two points lead to 
 a significantly simpler proof of quasi-invariance
than  the two-dimensional case \cite{OTz3}.

\begin{remark}\label{REM:new1}\rm

Following the discussion of Remark (iv) in Subsection \ref{SEC1:2}, 
one might attempt to
implement an analogous construction of weighted Gaussian measure in the case of 
NLW  with a higher order nonlinearity or in higher dimensions.
Higher order nonlinearities would result in a higher power of the regular
part of the renormalized energy, while the singular part would remain
quadratic, i.e.~$(D^s u)^2$. Thus, the construction of these measures
seems tractable. This is in sharp contrast with the construction of the $\Phi_3^{2n}$ measures, 
where higher order nonlinearities result in higher powers of distributions which makes the construction of such measures unclear (for $n \geq 3$).
On the other hand, higher dimensions in our case would result in a more singular quadratic part.

\end{remark}

\subsection{Statements of key results}
\label{SUBSEC:proof1}

In the remaining part of this paper, we fix $d = 3$.
In this subsection, we  introduce a new renormalized energy
and then state the key propositions in proving Theorem \ref{THM:NLW}.

We first introduce a new Gaussian measure, 
whose energy is
more suitable for analysis on NLW
(but still controls the zeroth frequency).
Define  a Gaussian measure $\nuu_{s} $ via the following  Karhunen-Lo\`eve expansions:
\begin{align}
\begin{split}
\label{series4}
u^\o(x) &= g_0(\omega) + \sum_{n \in \Z^3\setminus \{0\}} \frac{g_n(\o)}
{
(|n|^2+|n|^{2s+2})^{\frac{1}{2}}
}e^{in\cdot x},  \\
v^\o(x) &= \sum_{n \in \Z^3} 
\frac{h_n(\o)}{
(1+|n|^{2s})^{\frac{1}{2}}
}e^{in\cdot x}, 
\end{split}
\end{align}

\noi 
where  $\{ g_n \}_{n \in \Z^3}$ and $\{ h_n \}_{n \in \Z^3}$
are as in \eqref{series}.
Then, the formal density of 
$\nuu_s$ is given by 
\begin{align*}
d \nuu_s = Z_s^{-1} 
e^{-H_s(\vec{u})}d\vec{u}, 
\end{align*}

\noi
where 
\begin{equation}\label{lili}
H_s(\vec{u}) =  \frac 12 \bigg(\int_{\T^3}  u\bigg)^2 + \frac 12 \int_{\T^3}  |\nb u|^2 +
 \frac 12 \int_{\T^3} (D^{s+1} u)^2 
+    \frac 12 \int_{\T^3} v^2+ \frac 12  \int_{\T^3} (D^s v)^2.
\end{equation}

\begin{lemma} \label{LEM:equiv}
Let $s > \frac 34$. Then, the Gaussian measures $\muu_s$ and $\nuu_s$ are equivalent.
\end{lemma}

The proof of this lemma is based on a simple application of Kakutani's theorem \cite{Kakutani};
see the proof of Lemma 6.1 in \cite{OTz3} for details in the two-dimensional case.

\begin{remark}
\label{REM:zero_freq}\rm

The linear wave equation conserves the homogeneous Sobolev norm:
\begin{equation*}
\| \vec{u} \|_{\vec{\dot{H}}^{s+1}}^2 =  \int_{\T^3} (D^{s+1} u)^2 +  \int_{\T^3} (D^s v)^2.
\end{equation*}

\noi
Hence, we would like to work with Gaussian measures with formal density 
$e^{-\frac 12 \| u \|_{\vec{\dot{H}}^{s+1}}^2}$. 
These measures do not exist as probability measures since the zeroth frequency  is not controlled. 
This is the reason  we chose to include $g_0(\omega)$
in \eqref{series4}, giving rise to the first term in $H_s(\vec u)$ defined in \eqref{lili}.

As we see below, 
we add the truncated energy $E_N(\vec u)$ in \eqref{E2}
to construct  the full renormalized energy,
which explains the appearance of 
 the terms with $|\nabla u|^2$ and $v^2$ 
 in~\eqref{lili}.
 This addition of the truncated energy $E_N(\vec u)$
 allows us to include the quartic potential energy $\frac 14 \int u_N^4$
 without changing the time derivative
of the renormalized energy; see \eqref{H8}. 
We point out that  this quartic homogeneity plays an important role in the construction of the weighted Gaussian measure.  
\end{remark}

Given $N \in \N$, we  {\it redefine} the parameter $\s_N$, 
adapted to the new Gaussian measure $\nuu_s$,  by 
\begin{align}
\s_N \stackrel{\text{def}}{=}
\E_{\nuu_s}
\Big[ (D^s u_N)^2\Big] 
= \sum_{\substack{n \in \Z^3\\1\leq |n|\leq N}} 
\frac{|n|^{2s}}{|n|^2+|n|^{2s+2}}\sim  N
\longrightarrow \, \infty
\label{sigma}
\end{align}

\noi
as $N \to \infty$.
We also redefine the operator $Q_{s, N}$ in \eqref{H1c}
with this new definition of $\s_N$.
In the remaining part of this paper, 
we will use these new definitions for $\s_N$ and 
$Q_{s, N}$.

We now 
define the full renormalized  energy $E_{s, N}(\vec{u})$ by
\begin{align}
E_{s, N}(\vec{u})=\HH_{s, N}(\vec{u})+E_N(\vec{u}) + \frac 12 \bigg(\int_{\T^3}  u_N\bigg)^2 , 
\label{Es}
\end{align}

\noi
where $\HH_{s,N}$ is as in \eqref{H3} and $E_N$ is the truncated energy in \eqref{E2}.
\noi
Then, 
it follows from~\eqref{H3a} and   the conservation of 
the truncated energy that
\begin{align}
\begin{split}
\dt E_{s, N}(\vec{u}) 
&  =  
 3 \int_{\T^3}  \DD   v_N  u_N \\
& \hphantom{XX}
+ \sum_{\substack{ |\al|+|\be|+|\g|=  s\\
|\al|,|\be|,|\g|<s
}}
c_{\al,\be,\gamma}
\int_{\T^3}
D^sv_N\cdot \dd^\al u_N \cdot \dd^\be u_N \cdot \dd^\g u_N \\
& \hphantom{XX} +  \bigg(\int_{\T^3}  u_N\bigg) \bigg( \int_{\T^3} v_N \bigg)
\end{split}
\label{H8}
\end{align}

\noi
for any solution  $\vec{u}$ to the truncated NLW \eqref{NLW-sys_CP}.
Moreover, from \eqref{lili}, we have
\begin{align*}
E_{s, N}(\vec{u})=H_s(\vec{u})+ R_{s, N}(u), 
\end{align*}

\noi
where
\begin{align}
\begin{split}
R_{s, N}(u) 
& =  \frac 32  \int_{\T^3} \DD  u_N^2  +\frac{1}{4}\int_{\T^3}u_N^4\\
& =  \frac 32  \int_{\T^3} \Big( (D^s u_N)^2 - \s_N \Big)  u_N^2 
+\frac{1}{4}\int_{\T^3}u_N^4.
\end{split}
\label{K1}
\end{align}

We are now ready to state the two key ingredients 
for proving Theorem \ref{THM:NLW}:
(i) the construction of the weighted Gaussian measures
and (ii) the renormalized energy estimate.

Define the weighted Gaussian measure $\rhoo_{s, N}$ by
\begin{align}
\begin{split}
d\rhoo_{s, N}(\vec{u})
& = \ZZ_{s,N}^{-1}  e^{-R_{s, N}(u)} d \nuu_s(\vec{u}),
\end{split}
\label{K0}
\end{align}
\noi where $\ZZ_{s,N}$ is the normalization constant. 
The following proposition 
 establishes  uniform integrability
of the density $ e^{-R_{s, N}(u)}$ in~\eqref{K0}, 
which allows us to construct 
 the limiting weighted Gaussian measure $\rhoo_{s}$ by 
\begin{align*}
d\rhoo_{s}(\vec u) 
& = \ZZ_{s}^{-1} e^{-  R_s(u)} d \nuu_s(\vec u ),
\end{align*}

\noi
where $R_s(u)$ is a limit of $R_{s, N}(u)$; see Lemma \ref{LEM:QFT2}.

\begin{proposition}[Construction of the weighted Gaussian measure]\label{PROP:QFT0}
Let $ s > \frac 32$.
Then, 
the weighted Gaussian measures 
$\rhoo_{s, N}$ converges strongly 
to $\rhoo_{s}$.
Namely, we have
\begin{align*}
\lim_{N\to \infty} \rhoo_{s, N}(A) =   \rhoo_{s}(A)
\end{align*}

\noi
for  any measurable set 
$A \subset \H^\s(\T^3)$, $\s < s - \frac 12$. 
Moreover, given any finite $p\geq 1$, the sequence $\big\{ e^{-R_{s, N}(u)}\big\}_{N \in \N}$
and $e^{-R_s(u)}$ are uniformly bounded in $ L^p(\nuu_s)$. 
 As a consequence, $\rhoo_{s}$
is equivalent to 
$\nuu_s$.

\end{proposition}

Next, we state the key renormalized energy estimate, whose proof is deferred to 
the end of Section~\ref{SEC:ENERGY}.
Recall that 
 $B_R$  denotes the ball of radius $R>0$ in $\H^\s(\T^3)$
centered at the origin. 
We denote by $\Phi_N(t)$ the flow of  the truncated NLW dynamics \eqref{NLW-sys_CP}. 

\begin{proposition}[Renormalized energy estimate]\label{PROP:Nenergy}
Let $s\geq 4$ be an even integer.
Then,   given $R > 0$,  there is a constant $C = C(R)>0$ such that 
\begin{align}
\Bigg\{\int
\ind_{B_R}(\vec{u})\cdot 
\Big|\dt E_{s, N}(\pi_N\Phi_N(t)(\vec{u}))\vert_{t=0} \Big|^p d\nuu_s(\vec{u})\Bigg\}^{\frac{1}{p}}\leq Cp
\label{XX1}
\end{align}

\noi
for any finite $p\geq 1$ and any $N\in \N$.
\end{proposition}


Before we state  the main proposition
on  the evolution of the truncated measures~$\rhoo_{s, N}$, 
let us state  the following change-of-variable formula.
Given  $N \in \N$,  let 
$\EE_N = \pi_NL^2(\T^3)$
and  we endow  $\EE_N\times \EE_N$ with the Lebesgue measure $L_N$ 
as in Section~\ref{SEC:TOOL}.
Then, by viewing  the Gaussian measure  $\nuu_s$ as a product measure on $(\EE_N\times \EE_N)
\times (\EE_N\times \EE_N)^\perp$, 
we can write the truncated weighted Gaussian measure 
$\rhoo_{s, N}$ defined in \eqref{K0} as
\begin{align}
\begin{split}
d\rhoo_{s, N}(\u )
&  =  \ZZ_{s,N}^{-1} \,
e^{-  R_{s, N}(\pi_{N}u)} d \nuu_s(\u) ,
\\
& =  \hat{Z}_{s,N}^{-1} \,
e^{-  E_{s, N}(\pi_{N}\u)}\,  dL_N\otimes d  \nuu^\perp_{s; N} (\u),
\end{split}
\label{XY1}
\end{align}

\noi
where $ \hat{Z}_{s,N}$ denotes the normalization constant
and
$\nuu^\perp_{s;N}$ denotes the marginal Gaussian  measure of $\nuu_s$
on $ (\EE_N\times \EE_N)^\perp$.
 Then, we have the following change-of-variable formula.

\begin{lemma}\label{LEM:cov}
Let $s>\frac 32$ and $N \in \N$.
Then, we have 
$$
\rhoo_{s,N}(\Phi_N(t)(A))=\hat{Z}_{s,N}^{-1} 
\int_A
e^{-  E_{s, N}(\pi_{N}\Phi_N(t)(\u))}
 \, dL_N\otimes d \nuu^\perp_{s;N} (\u)
$$

\noi
for any $t \in \R$ and any measurable set
 $A\subset \H^\sigma(\T^3)$ with $\s< s - \frac 12$.
\end{lemma}

The proof of Lemma~\ref{LEM:cov} is 
based on (i) the invariance of the Lebesgue measure $L_N$
under (the low frequency part of) the truncated NLW dynamics $\pi_N \Phi_N(t)$, 
(ii)  the conservation of the truncated energy $E_N(\u)$ under $\Phi_N(t)$
and (iii)  the bijectivity of the solution map $\Phi_N(t)$.
As it follows from similar considerations presented in~\cite{Tzvet, OTz},  
 we omit details of the proof.

We now state and prove the main proposition,
essentially establishing the differential inequality \eqref{measure_diffeqn}.
This proposition allows us to control the growth of the pushforward 
measure  $\rhoo_{s, N}(\Phi_N(t) (A))$
of a given measurable set $A  \subset \H^\sigma(\T^3)$
uniformly in $N \in \N$, 
provided that the set $A$ lies in the ball
$ B_R \subset \H^\sigma(\T^3)$ of radius $R>0$.
Namely, it only provides a  {\it set-dependent}
 control.
This dependence on $R>0$, however, does not cause any trouble
in establishing quasi-invariance of 
 the Gaussian measure $\nuu_s$ (and hence of $\muu_s$).

\begin{proposition}\label{PROP:meas2}
	Let $s \geq 4$ be an even integer
	and  $\s\in \big(1,s-\frac 12\big)$.
	Then, given  $R > 0$ and $T > 0$, 
	there exists $C_{R, T}>0$ such that 
	\begin{align*}
	\frac{d}{dt} \rhoo_{s, N}(\Phi_N(t) (A))
	\leq C_{R, T}\cdot  p \, \big\{ \rhoo_{s, N} (\Phi_N(t)(A))\big\}^{1-\frac 1p}
	\end{align*}
	\noi
	for any $p \geq 2$, 
	any $N \in \N$, any $t \in [0, T]$, 
	and  any measurable set 
	$A \subset B_R \subset \H^\sigma(\T^3)$.
\end{proposition}

In \cite{OTz3}, there is an analogous statement, 
controlling the evolution of the truncated measures
(without the restriction on $B_R$);
see \cite[Lemma 5.2]{OTz3}.
The main idea of the proof of Lemma 5.2 in \cite{OTz3}
is to reduce
the analysis to that at $t = 0$, which provides access
to the random distributions in \eqref{series4}.
On the other hand, 
the main idea in \cite{PTV} at this step is to use the 
{\it deterministic}  control \eqref{growth1} on the growth of solutions.
In the following, 
 we combine both of these ideas, 
 thus introducing a hybrid argument
 which works more effectively than each of the two methods.

\begin{proof}
	Fix $R, T > 0$ and $t_0\in[0, T]$. Let $A \subset B_R$
	be a measurable set in $\H^\s(\T^3)$.
Using the flow property of $\Phi_N(t)$, we have
\begin{align*}
\frac{d}{dt} \rhoo_{s, N}(\Phi_N(t) (A))\bigg|_{t=t_0}
& =
 \ZZ_{s,N}^{-1} 
\frac{d}{dt}
\int_{\Phi_N(t)(A)}e^{-  R_{s, N}(\pi_{N}u)} d \nuu_s(\vec u)\bigg|_{t=t_0}
\\
& =
\ZZ_{s,N}^{-1} 
\frac{d}{dt}
\int_{\Phi_N(t)(\Phi_{N}(t_0)(A))}e^{-  R_{s, N}(\pi_{N}u)} d \nuu_s(\vec u)\bigg|_{t=0}.
\end{align*}

\noi
The change-of-variable argument (Lemma \ref{LEM:cov}),
\eqref{XY1},  
and the growth bound \eqref{growth1}  in Lemma \ref{LEM:CP_APPRX}  yield
\begin{align} \label{eq growth}
\begin{split}
\frac{d}{dt} & \rhoo_{s, N}  (\Phi_N  (t) (A))\bigg|_{t=t_0}
\\
& =
\hat{Z}_{s,N}^{-1} 
\frac{d}{dt}
\int_{\Phi_{N}(t_0)(A)}
e^{-  E_{s, N}(\pi_{N}\Phi_N(t)(u,v))}
  dL_N\otimes  d \nuu^\perp_{s;N}
\bigg|_{t=0}
\\
& =- \ZZ_{s,N}^{-1} 
\int_{\Phi_{N}(t_0)(A)}
\dt E_{s, N}(\pi_N\Phi_N(t)(\u)) \big|_{t=0}\,
e^{-  R_{s, N}(\pi_{N}u)} d \nuu_s(\u) \\
& \leq \ZZ_{s,N}^{-1} 
\int_{B_{C(R, T)}}
\Big|\dt E_{s, N}(\pi_N\Phi_N(t)(\u)) \big|_{t=0}\Big|\,
e^{-  R_{s, N}(\pi_{N}u)} d \nuu_s(\u) .
\end{split}
\end{align}

\noi
Then, from  H\"older's inequality, we obtain
\begin{align*}
\frac{d}{dt} \rhoo_{s, N}(\Phi_N(t) (A))\bigg|_{t=t_0}
& \leq 
\Big\| \ind_{B_{C(R, T)}}(\u) \cdot 
\dt E_{s, N}(\pi_N\Phi_N(t)(\u))\big|_{t=0}
\Big\|_{L^p(\rhoo_{s, N})} \\
& \hphantom{X}
\times
\big\{\rhoo_{s, N}(\Phi_N(t_0) (A))\big\}^{1-\frac{1}{p}}.
\end{align*}

\noi
Finally, by Cauchy-Schwarz inequality
 together with the uniform exponential moment
 bound on $R_{s, N}(u)$  in Proposition~\ref{PROP:QFT0}
 and Proposition \ref{PROP:Nenergy}, we obtain 
 \begin{align}
\begin{split}
\Big\| & \ind_{B_{C(R, T)}}  (u, v) \cdot  \dt E_{s, N}
 (\pi_N\Phi_N(t)(\u))\big|_{t=0}
\Big\|_{L^p(\rhoo_{s, N})} \\
& \leq \ZZ_{s,N}^{-\frac{1}{p}}\,
\Big\|
\ind_{B_{C(R, T)}}(\u) \cdot 
\dt E_{s, N}(\pi_N\Phi_N(t)(\u))\big|_{t=0}
\Big\|_{L^{2p}(\nuu_{s, N})}
 \Big\| e^{-  R_{s, N}(u)} \Big\|_{L^{2}(\nuu_s)}^\frac{1}{p}\\
&  \leq C_{R, T} \cdot p.
\end{split}
\label{N1}
 \end{align}
 
 \noi
Here, we used the boundedness of 
$ \ZZ_{s,N}^{-1}$, uniformly in $N \in \N$
(recall that $ \ZZ_{s,N}\to  \ZZ_{s}>0$ as $N \to \infty$).
This completes the proof of Proposition~\ref{PROP:meas2}.
\end{proof}

\subsection{Proof of Theorem \ref{THM:NLW}}
\label{SUBSEC:proof2}

We conclude this section by presenting the proof of Theorem~\ref{THM:NLW}.
Our aim is to show that
for each {\it fixed} $R>0$, we have
\begin{align}
\nuu_{s}(A) = 0
\qquad 
\text{implies}
\qquad \nuu_{s}\big(\Phi(t) (A)\big) =0
\label{N2}
\end{align}

\noi
for any measurable set $A \subset B_R \subset \H^\s(\T^3)$, $\sigma \in (1, s - \frac 12)$
and any $t >0$.\footnote{In view of  the time reversibility of the equation \eqref{NLW}, 
	it suffices to consider positive times.}
Since the choice of $R>0$ is arbitrary, 
this yields quasi-invariance of $\nuu_s$ 
under the NLW dynamics.
Then, we invoke Lemma \ref{LEM:equiv}
to conclude quasi-invariance of $\muu_s$ (Theorem \ref{THM:NLW}).

Arguing as in \cite{OTz3}, 
Proposition~\ref{PROP:meas2}
allows us to establish 
quasi-invariance of the truncated weighted Gaussian measures $\rhoo_{s, N}$
with the uniform control in $N\in \N$
(but with  dependence on $R>0$). 
See Proposition 5.3  in \cite{OTz3}.
By the approximation property of the truncated NLW dynamics (Lemma \ref{LEM:CP_APPRX}\,(ii))
and the strong convergence of $\rhoo_{s, N}$ to $\rhoo_{s}$
(Proposition~\ref{PROP:QFT0}),
we can upgrade this to the $N = \infty$ case, 
thus establishing quasi-invariance of 
the untruncated 
weighted Gaussian measure $\rhoo_{s}$ 
under the NLW dynamics.
See  Lemma 5.5 in \cite{OTz3} for the proof.

\begin{lemma}\label{LEM:meas4}
Given  any $R > 0$,  there exists $t_* = t_*(R) \in (0, 1]$ such that 
for any $\eps > 0$,  
there exists $\dl > 0$ 
with the following property;
	if a measurable set $A \subset B_R \subset \H^\s(\T^3)$, $\s\in \big(1,s- \frac 12\big)$ satisfies
	$$
	\rhoo_{s} ( A)< \dl,  
	$$

\noi
then we have 
	$$
	\rhoo_{s} (\Phi(t) (A)) < \eps
	$$
	for any $t \in [0, t_*]$.
\end{lemma}

Finally,  we establish \eqref{N2}
by exploiting the mutual absolute continuity 
between $\rhoo_{s}$ and $\nuu_{s}$
for each fixed  $R>0$. Let $A \subset B_R$ be such that $\nuu_s(A) = 0$. 
By the mutual absolute continuity 
of $\nuu_{s}$ and $\rhoo_{s}$, 
we have
\[\rhoo_{s}(A) = 0.\]

\noi
Now, fix a target time $T > 0$
and let $C(R, T)$ be as in Lemma \ref{LEM:CP_APPRX} (i).
Namely, we have 
\begin{align}
\Phi(t)(A) \subset B_{C(R, T)}
\label{meas10}
\end{align}

\noi
for all $t \in [0, T]$.
Then, by applying Lemma \ref{LEM:meas4}
with $R$ replaced by $C(R, T)$, we obtain
\begin{align}
\rhoo_{s}(\Phi(t) (A)) = 0
\label{meas9}
\end{align}

\noi
for $t \in [0, t_*]$, 
where $t_* = t_*(C(R, T))$.
In view of \eqref{meas10}, 
we can iterate this argument
and 
 conclude that \eqref{meas9} holds for any $t \in [0, T] $.
Since the choice of $T>0$ was arbitrary, 
we obtain \eqref{meas9} for any $t>0$.
Finally, by invoking the mutual absolute continuity 
of $\nuu_{s}$ and $\rhoo_{s}$ once again, 
we have
\[\nuu_{s}(\Phi(t) (A)) = 0\]

\noi
for any $t > 0$. This proves \eqref{N2}
and hence Theorem \ref{THM:NLW}.

\begin{remark}\label{REM:loss}\rm
  While this new hybrid argument allows us to establish quasi-invariance of the Gaussian measure $\nuu_s$
  (and hence $\muu_s$)
  under the NLW dynamics even in the three-dimensional case,
it does not provide as good of a quantitative bound as
the two-dimensional argument.
For example, in the two-dimensional case,
the argument in \cite{OTz3} yielded
\begin{align}
\rhoo_{s} (\Phi(t) (A)) \les
\big(\rhoo_{s} ( A) \big)^{\frac {1}{c^{1+|t|}}}
\label{Nbound2}
\end{align}
for a weighted Gaussian measure $\rhoo_{s, r}$
with an energy cutoff $\ind_{\{E(u, v) \leq r\}}$,
where $c = c(r) >0$; see Remark 5.6 in \cite{OTz3}.
Our present understanding does not provide an 
analogous bound to \eqref{Nbound2} in three dimensions.

\end{remark}

\section{Construction of the weighted Gaussian measure}\label{SEC:QFT}

In this section, we prove Proposition \ref{PROP:QFT0}
by establishing  
 uniform integrability of the densities $R_{s, N}(u)$
of  the  weighted Gaussian measures 
$\rhoo_{s, N}$ in \eqref{K0}. 
In Subsection~\ref{SUBSEC:QFT1}, 
we first prove some regularity properties
of random distributions (Proposition~\ref{PROP:random_distr})
and then 
 the $L^p$-convergence of $R_{s, N}(u)$ in \eqref{K1}.
We split the proof of the main result (Proposition~\ref{PROP:QFT})
into two parts.
In Subsection~\ref{SUBSEC:var}, 
we follow the argument by 
Barashkov and Gubinelli~\cite{BG}
and
express the partition function $\ZZ_{s, N}$
in terms of a minimization problem
 involving a stochastic control problem
 (Proposition \ref{PROP:var}).
 In Subsection~\ref{SUBSEC:const2}, 
 we then study the minimization problem
 and 
establish boundedness of the 
partition function $\ZZ_{s, N}$,  uniformly in $N \in \N$.

Let $N \geq 1$. Recall that $\rhoo_{s, N}$ has density $e^{-  R_{s, N}(u)}$ with respect to $\nuu_s$.
In particular, note that  the non-Gaussian part of $\rhoo_{s,N}$ depends only on $u$. This motivates the following reduction; define $H^{(1)}_s(u)$ and $H^{(2)}_s(v)$ by  
\begin{align*}
H^{(1)}_s(u)
& =   \frac 12 \bigg(\int_{\T^3}  u\bigg)^2 + \frac 12 \int_{\T^3}  |\nb u|^2 + \frac 12 \int_{\T^3} (D^{s+1} u)^2,
\\
H^{(2)}_s(v)
& =   \frac 12 \int_{\T^3} v^2+\frac 12  \int_{\T^3} (D^s v)^2 .
\end{align*}

\noi
Then, define Gaussian measures
$\nu_s^{(j)}$, $j = 1, 2$, 
with formal densities:
\begin{align*}
d \nu_s^{(1)} = Z_{1, s}^{-1} e^{-H_s^{(1)}(u)} du
\qquad \text{and}\qquad
d \nu_s^{(2)} = Z_{2, s}^{-1} e^{-H_s^{(2)}(v)} dv.
\end{align*}

\noi
Since $H_s(\u) = H_s(u, v)$ in \eqref{lili} is now written as 
\[ H_s(\u) = H^{(1)}_s(u)+H^{(2)}_s(v),\]

\noi
the Gaussian measure $\nuu_s$ 
can be rewritten as 
\begin{align}
 d \nuu_s(\vec{u}) = d\nu_s^{(1)}(u) \otimes d\nu_{s}^{(2)}(v).
 \label{nu3}
\end{align}

\noi
From decomposition~\eqref{nu3}, we have
\begin{align*}
d\rhoo_{s, N}(\vec{u}) 
& = d\rho_{s, N} (u) \otimes d\nu_{s}^{(2)}(v),
\end{align*}

\noi
where $\rho_{s, N}$ is given by 
\begin{equation*}
d\rho_{s, N} (u)=  \ZZ_{s, N}^{-1} e^{-R_{s, N}(u)} d\nu^{(1)}_{s}(u).
\end{equation*}

\noi
The partition function  $\ZZ_{s, N}$ is now expressed as 
\begin{equation}
\ZZ_{s, N}=  \int e^{-R_{s, N}(u)} d\nu_{s}^{(1)}(u).
\label{Z1}
\end{equation}

In the following, 
we denote  $\nu_{s}^{(1)}$ by $\nu_s$
and  prove various statements in terms of $\nu_s$
but they can be trivially upgraded to the corresponding statement for $\nuu_s$.

\begin{lemma}\label{LEM:QFT2}
Let $ s > \frac 32$.
Then,  given any finite $p < \infty$, 
$R_{s, N} $ defined in \eqref{K1} converges to some $R_s$ in $L^p( \nu_s)$ as $N \to \infty$.

\end{lemma}

The goal of this section is to
prove the following proposition on 
 uniform  (in $N \in \N$)  integrability of the density 
$e^{-  R_{s, N}(u)}$ for $\rhoo_{s, N}$,
which allows us to   construct the limiting measure $\rhoo_{s}$.
 As a consequence of our construction, 
 the weighted Gaussian measure $\rhoo_s$ is equivalent to $\nuu_s$ (and hence to $\muu_s$
 in view of Lemma \ref{LEM:equiv}).

\begin{proposition}\label{PROP:QFT}
Let $ s > \frac 32$.
Then, given any finite $ p < \infty$, 
there exists $C_p > 0$ such that 
\begin{equation}
\sup_{N\in \N} \Big\| e^{- R_{s, N}(u)}\Big\|_{L^p(\nu_s)}
\leq C_p  < \infty.
\label{exp1}
\end{equation}

\noi
Moreover, we have
\begin{equation}\label{exp2}
\lim_{N\rightarrow\infty}e^{-  R_{s, N}(u)}=e^{-  R_s(u)}
\qquad \text{in } L^p(\nu_s).
\end{equation}

\end{proposition}

While the first part of Proposition \ref{PROP:QFT0}
follows  from  Proposition \ref{PROP:QFT} with $p = 1$, 
we need to have the uniform bound \eqref{exp1}
for some $p > 1$ for the proof of Proposition~\ref{PROP:meas2}. 
See~\eqref{N1}.
Note that this requirement on a higher integrability for some $p > 1$
is analogous to the situation in Bourgain's construction on
invariant Gibbs measures for Hamiltonian PDEs \cite{BO94},
where, as in \eqref{N1}, the analysis of the weighted Gaussian measure
needs to be reduced to that of the underlying Gaussian measure by Cauchy-Schwarz inequality.
Since the argument is identical for any $p \geq 1$, we only present details for the case $p =1$. We point out that the $L^p$-convergence \eqref{exp2}
is a consequence of 
 the uniform exponential moment bound~\eqref{exp1} and the softer convergence in measure (as a consequence of Lemma~\ref{LEM:QFT2}). 
See Remark~3.8 in \cite{TZ2}. 
Therefore, we focus on proving  the uniform bound \eqref{exp1}. 

In the next subsection, we prove Lemma \ref{LEM:QFT2}. 
The subsequent subsections are devoted to the proof of  Proposition \ref{PROP:QFT}.

\subsection{Regularity of  random distributions}
\label{SUBSEC:QFT1}

Let $u$ be distributed according to $\nu_s$
and $Q_{s, N}$ be as in \eqref{H1c}
with $\s_N$  in \eqref{sigma}.
In this case, 
 we have
\begin{align}
 :\!(D^s u_N)^2\!: \,   =  \DD, 
\label{XX2}
\end{align}

\noi
where the left-hand side is the standard notation for the Wick renormalization.

We first state and prove the regularity properties of 
(products of) certain random distributions.
The proof of Lemma~\ref{LEM:QFT2} is presented at the end of this subsection.

\begin{proposition} \label{PROP:random_distr}
Let $ s \geq 1$ and  $\varepsilon > 0$. Then, there exists 
$C=C(s,\varepsilon)>0$ such that for any $N \in \N$ and any $2 \leq p < \infty$, we have
\begin{align}
\label{rd1}\| :\!(D^s u_N)^2\!: \|_{L^p(\nu_s, \, \C^{-1-\varepsilon})} &\leq Cp,  \\
\label{rd2}\| \dd^\kk v_N \, \partial^\alpha u_N \|_{L^p(\vec \nu_s (u, v),\,  \C^{-1-\varepsilon})} &\leq Cp  
\qquad \text{for $|\kk|=s-1$ and $|\alpha| = s$,} \\
\label{rd3}\| \dd^\kk v_N \, \partial^\alpha u_N \|_{L^p(\vec \nu_s(u, v), \, \C^{-\frac{1}{2}-\varepsilon})} &\leq Cp  
\qquad \text{for $|\kk|=s-1$ and $|\alpha| \leq  s-1$}, 	
\end{align}

\noi
where $u_N = \pi_N u $ and $v_N  =\pi_N v$.
Moreover, as $N \rightarrow \infty$,  the sequences above converge to limits denoted by 
$:\!(D^s u)^2\!:$ and $\dd^\kk v \, \dd^\alpha u$ with respect to the same topologies. 
\end{proposition}

We will also use this proposition 
in proving the renormalized energy estimate in Section~\ref{SEC:ENERGY}.

\begin{proof}

We only prove \eqref{rd1} in the following.
 The other estimates \eqref{rd2}
 and  \eqref{rd3} follow in a similar manner, with the simplification that no renormalization is needed due to the independence of $u$ and $v$ under $\nuu_s$.
The regularity $-1 - \eps$ 
in \eqref{rd2} is naturally expected
in view of the regularities $< - \frac 12$
for each of 
$ \dd^\kk v_N $ and $ \partial^\alpha u_N $.
A similar comment applies to~\eqref{rd3},
where the regularity of $\dd^\kk v$ is less than $-\frac12$.

Noting that 
\[ \frac{|n|^s}{(|n|^2 + |n|^{2s + 2})^\frac{1}{2}} \les \frac{1}{\jb{n}}\]

\noi
for any $n \in \Z^3\setminus\{0\}$, 
it follows from  the Karhunen-Lo\`eve expansion \eqref{series4} that 
\begin{align} \label{Xseries}
\begin{split}
\E_{\nu_s} \Big[ \big| \F\big\{ :\!(D^s u_N)^2 \!:\big\} (n) \big|^2  \Big] 
& \les \sum_{\substack{n_1, n_2 \in \Z^3\\|n_j| \leq N}}
\frac{\big|\E [ g_{n_1}g_{n-n_1}g_{-n_2}g_{-n+n_2} ]\big|}
{\jb{n_1} \jb{n - n_1}\jb{n_2}\jb{n - n_2}} \ind_{\{n \neq 0 \}} \\
&\hphantom{X}
+  \sum_{\substack{n_1, n_2 \in \Z^3\\|n_j| \leq N}}
\frac{\big|\E \big[ (|g_{n_1}|^2 - 1)(|g_{n_2}|^2 - 1) \big]\big|}{\jb{ n_1}^2 \jb{ n_2}^2} \ind_{\{n = 0 \}}
\end{split}
\end{align}

\noi
for any $n \in \Z^3$,
where $\F$ denotes Fourier transform.
In the first sum on the right-hand side of  $\eqref{Xseries}$, 
we note that due to the independence (modulo the conjugates) of the $g_n$'s and by Wick's theorem, 
all non-vanishing terms must satisfy $n_1 = n_2$ or $n_1 = n - n_2$. 
Thus, we obtain
\begin{equation} 
 \sum_{\substack{n_1, n_2 \in \Z^3\\|n_j| \leq N}}
\frac{\big|\E [ g_{n_1}g_{n-n_1}g_{-n_2}g_{-n+n_2} ]\big|}
{\jb{n_1} \jb{n - n_1}\jb{n_2}\jb{n - n_2}} \ind_{\{n \neq 0 \}} \\
\les \sum_{n_1 \in \Z^3} \frac{1}{\jb{n_1}^2 \jb{n - n_1}^2} 
\les \frac{1}{\jb{n}}
\label{series2}
\end{equation}

\noi
uniformly in $N \in \N$, 
where in the last inequality we used a standard result on discrete convolutions (see Lemma~4.2 in~\cite{MWX}). 
In the second sum on the right-hand side of  \eqref{Xseries}, 
we note that, by Wick's theorem, 
the contribution from  $|n_1| \neq |n_2|$ vanishes.
Thus, we obtain
\begin{align}
 \sum_{\substack{n_1, n_2 \in \Z^3\\|n_j| \leq N}}
\frac{\big|\E \big[ (|g_{n_1}|^2 - 1)(|g_{n_2}|^2 - 1) \big]\big|}{\jb{ n_1}^2 \jb{ n_2}^2} \ind_{\{n = 0 \}}
\les 1,
\label{series3}
\end{align}

\noi
uniformly in $N \in \N$.
Putting \eqref{series2} and \eqref{series3} together, we obtain
\begin{equation*}
\E \Big[ \big| \F \big\{ :\!(D^s u_N)^2 \!:\big\} (n) \big|^2  \Big]
 \les \frac{1}{\jb{n}} 
\end{equation*}
	
\noi
for any $n \in \Z^3$ and $N \in \N$.

By a similar computation, we have
\begin{equation*} 
\E \Big[ \big| \F \big\{ :\!(D^s u_N)^2 \!: - :\!(D^s u_M)^2 \!:\big\} (n) \big|^2  \Big]
 \les \frac{1}{N^\ta \jb{n}^{1-\ta}}
\end{equation*}
	
\noi
for any $n \in \Z^3$,  any $M \geq N \geq 1$,
and $\ta \in [0, 1]$.
Note that  $:\!(D^s u_N)^2\!:$ lies in the second homogeneous Wiener chaos
$\mathcal H_2$.
Hence, by Lemma~\ref{LEM:MWX} with $\ta >0$ sufficiently small, 
we conclude that 
 $:\!(D^s u_N)^2\!:$ converges to some
 $:\!(D^s u)^2\!:$
 in $L^p(\nu_s; \C^{-1-\eps}(\T^3))$ for any finite $p \geq 2$. 
  \end{proof}

We now present the proof of Lemma  \ref{LEM:QFT2}.

\begin{proof}[Proof of Lemma \ref{LEM:QFT2}]
For  $s > \frac 32$, Lemma~\ref{LEM:MWX}
 implies $u_N$ converges to $u$ in  $L^p(\nu_s; \C^\s)$ for any finite $p \geq 2$ 
 and  any $\s < s- \frac 12$.
In the following, we choose $\s > 0$ sufficiently close to $s - \frac 12$.
 Then, by the algebra property \eqref{alge}, 
 we see that $u_N^2$ (and $u_N^4$, respectively) converges to  $u^2$ 
 (and $u^4$, respectively) 
 in  $L^p(\nu_s;\C^\s)$ for any finite $p \geq 2$. 
 
Proposition \ref{PROP:random_distr} asserts that 
$:\!(D^s u_N)^2\!:$ converges to $:\!(D^s u)^2\!: \, \in L^p(\nu_s, \C^{-1-\varepsilon}(\T^3))$ 
for any $\varepsilon > 0$. 
Recall from \eqref{prod2}  that the bilinear multiplication map from $\C^{s_1} \times \C^{s_2}$ 
to $\C^{s_1}$ is a continuous operation for $s_1 < 0 < s_2$ such that $s_1 + s_2 >0$. 
Therefore, by choosing $\s > 1 + \eps$ (which is possible since $s > \frac 32$), 
we conclude that 
\[ :\!(D^s u)^2\!:u^2 = \lim_{N \to \infty} :\!(D^s u_N)^2\!: u_N^2\]

\noi
exists as an element in $L^p(\nu_s; \C^{-1-\varepsilon}(\T^3))$ for all finite $p \geq 2$. 
This means that
\begin{equation}
\label{qft2_pf_1}
\frac 32 :\!(D^s u)^2\!: u^2 + \frac 14 u^4 \in L^p(\nu_s, \C^{-1-\varepsilon}(\T^3)).
\end{equation}
Lemma \ref{LEM:QFT2} then follows from \eqref{qft2_pf_1}. 
\end{proof}

\subsection{Variational formulation}
\label{SUBSEC:var}

In this subsection, we follow the argument in \cite{BG} 
and derive a variational formula for the normalization constant
$\ZZ_{s, N}$ in \eqref{Z1}, which is based on a well-known representation of the classical Gibbs variational principle on the Wiener space \cite[Proposition 4.5.1]{DE}.

Given small $\eps >0$, 
let $\O_\eps = C(\R_+, \C^{-\frac 32-\eps}(\T^3))$ 
equipped with its Borel $\s$-algebra.
Denote by\footnote{In the remaining part
of this section, we use the standard notation in stochastic analysis where subscripts denote parameters for stochastic processes.} 
 $\{X_t\}$  the coordinate process on $\O_\eps$ 
 and consider the probability measure $\PP$ that makes 
 $\{X_t\}$ a cylindrical Brownian motion in $L^2(\T^3)$.
 Namely, we have
\begin{align*}
X_t = \sum_{n \in \Z^3} B_t^n e^{i n \cdot x},
\end{align*}

\noi
where  
$\{B_t^n\}_{n \in \Z^3}$ is a sequence of independent complex-valued\footnote{We normalize
$B_t^n$ so that $\text{Var}(B_t^n) = t$.  Moreover, we impose that $B^0_t$ is real-valued.}  Brownian motions such that 
$\cj{B_t^n}= B^{-n}_t$, $n \in \Z^3$.
Then, define a centered Gaussian process $\{ Y_t\}$
by 
\begin{align}
Y_t
=  \J^{-s-1}X_t
 \stackrel{\text{def}}{=} B^0_t + \sum_{n \in \Z^3\setminus \{0\}} \frac{B^n_t}{(|n|^2+|n|^{2s+2})^{\frac{1}{2}}}
 e^{ in \cdot x}.
\label{Z2}
\end{align}

\noi
Then, in view of \eqref{series4}, 
we have $\Law_{\PP}(Y_1) = \nu_s$. 
By truncating the sum in \eqref{Z2}, 
we also define the truncated process $Y_t^N = \pi_N Y_t$ 
with the property $\Law_{\PP}(Y_1^N) = \Law_{\nu_s}(\pi_N u)$. 
Note that we have $\E [(D^s Y_1^N)^2] = \s_N$,
where $\s_N$ is as in \eqref{sigma}.
For simplicity of notations, 
 we suppress dependence on $N \in \N$ when it is clear from the context.

Let $\Ha$ denote the space of progressively measurable processes that belong to
$L^2([0,1]; L^2(\T^3))$, $\PP$-almost surely. 
We say that an element $\dr$ of $\Ha$ is a {\it drift}.
Given a drift $\dr \in \Ha$, 
we  define the measure $\Q^\dr$ whose Radon-Nikodym derivative 
with respect to $\PP$
is given by the following stochastic exponential:
\begin{equation}
\frac{d\Q^\dr}{d\PP} = e^{\int_0^1 \jb{\dr_t,  dX_t} - \frac{1}{2} \int_0^1 \| \dr_t \|_{L^2_x}^2dt}.
\label{Z2a}
\end{equation}

\noi
Here, $\jb{\cdot, \cdot}$ denotes the inner product on $L^2(\T^3)$.
Then, by letting  $\Hc$ denote the space of drifts such that $\Q^\dr(\O_\eps) = 1$,
it follows from Girsanov's theorem (\cite[Theorem 10.14]{DZ} and \cite[Theorems 1.4 and 1.7 in Chapter VIII]{RV})
that the process $X_t$ is a semimartingale under $\Q^\dr$ 
with a decomposition:
\begin{align}
 X_t = X_t^\dr + \int_0^t \dr_{t'}dt',
\label{Z3}
\end{align}

\noi
 where $X_t^\dr$ is now a cylindrical Brownian motion in $L^2(\T^3)$ under
 the new measure $\Q^\dr$.
From~\eqref{Z3}, 
we also obtain the decomposition:
\begin{align}
 Y_t = Y_t^\dr + I_t(\dr), 
\label{Z4}
\end{align}

\noi
where $Y_t^\dr = \J^{-s-1} X_t^\dr$ and $I_t(\dr) = \int_0^t \J^{-s-1} \dr_{t'} dt'$.
In the following, we use $\E$ to denote an expectation 
with respect to $\PP$,
while we use $\E_\Q$ for an expectation with respect to some other probability measure $\Q$.

Before proceeding further, let us recall the following 
estimate (\cite[Lemma 2.6]{Fol}):
\begin{align}
\int_0^1 \| \dr_t \|_{L^2_x}^2dt 
\leq 2 H(\Q^\dr|\PP), 
\label{Fol1}
\end{align}

\noi
where 
$H(\Q^\dr|\PP)$
denotes 
the relative entropy of $\Q^\dr$ with respect to $\PP$ defined by 
\[ H(\Q^\dr|\PP) = \E_{\Q^\dr}\bigg[\log \frac{d\Q^\dr}{d\PP}\bigg]
= \E\bigg[\frac{d\Q^\dr}{d\PP}\log \frac{d\Q^\dr}{d\PP}\bigg].\]

\noi
With the notations introduced above, 
we have the following
variational characterization
of the partition function $\ZZ_{s, N}$
defined in \eqref{Z1}.

\begin{proposition} \label{PROP:var}
For any $N \in \N$, we have
\begin{equation} \label{var}
- \log \ZZ_{s,N} = \inf_{\dr \in \Hc} \E_{\Q^\dr} 
\bigg[ R_{s, N}(Y_1^\dr + I_1(\dr)) + \frac{1}{2} \int_0^1 \| \dr_t \|_{L^2_x}^2 dt \bigg].
\end{equation}

\end{proposition}

\begin{proof}
As a preliminary step, we first derive bounds
on $\ZZ_{s,N}$ and 
\[\E \bigg[ \frac{e^{-R_{s, N}(Y_1)}}{\ZZ_{s,N}} \log \bigg( \frac{e^{-R_{s, N}(Y_1)}}{\ZZ_{s,N}} \bigg)  \bigg].\] 
	
\noi
Note that these bounds imply 
that the measure $\frac{e^{-R_{s, N}(Y_1)}}{\ZZ_{s,N}}d\PP$ has a finite relative entropy with respect to $\PP$.

From \eqref{Z1},  Jensen's inequality, 
and \eqref{K1}, 
there exists finite $C(N) > 0$ such that
\begin{equation} 
\ZZ_{s,N} \geq e^{-\E[ R_{s, N}(Y_1) ]} 
\geq e^{-\E \big[ \frac 32 \int  (D^s Y^N_1)^2(Y_1^N)^2  +  \frac 14\int (Y_1^N)^4 \big]} \geq C(N).
\label{jensen1}
\end{equation}

\noi
In view of the following pointwise lower bound:
\begin{align}
\frac 32 (D^s Y_1^N)^2   (Y_1^N)^2 & -  \frac 32 \s_N (Y_1^N)^2 +  \frac 14(Y_1^N)^4
 \geq -\frac 32 \s_N     (Y_1^N)^2 +  \frac 14(Y_1^N)^4\notag \\
&   \geq -\frac 92 \s_N^2 +  \frac{1}{8}(Y_1^N)^4
 \geq - C(N) >-\infty, 
\label{Z5}
\end{align}

\noi
it follows from 
 $\eqref{jensen1}$, 
Cauchy's inequality, 
and  
Lemma \ref{LEM:QFT2}
that there exists finite $C(N) > 0$ such that
\begin{align}
\label{rent} 
\begin{split}
\E \bigg[ \frac{e^{-R_{s, N}(Y_1)}}{\ZZ_{s,N}} \log \bigg( \frac{e^{-R_{s, N}(Y_1)}}{\ZZ_{s,N}} \bigg)  \bigg] 
& \leq C(N) \E \Big[ e^{-R_{s, N}(Y_1)} \big(1 + \log e^{-R_{s, N}(Y_1)}\big)\Big] \\
& \leq C(N) \E \Big[ e^{-2R_{s, N}(Y_1)}+  |R_{s, N}(Y_1)|^2  + 1\Big]  \\
& \leq C(N) < \infty. 
\end{split}
\end{align}

Now, fix $\dr \in \Hc$. We show that
\begin{equation} \label{vform_aim}
- \log \ZZ_{s,N} \leq \E_{\Q^\dr} 
\bigg[ R_{s, N}(Y_1^\dr + I_1(\dr)) + \frac{1}{2} \int_0^1 \| \dr_t \|_{L^2_x}^2 dt \bigg].
\end{equation}  

\noi	
Suppose that 	
$\E_{\Q^\dr} \Big[ \int_0^1 \| \dr_t \|_{L^2_x}^2 dt \Big] = \infty$.
Then, $\eqref{vform_aim}$ holds trivially since
it follows from 
the decomposition \eqref{Z4} of $Y_t$ under $\Q^\dr$
and 
Cauchy's inequality
with Lemma \ref{LEM:QFT2},  \eqref{jensen1}, and~\eqref{Z5}
that
\begin{equation*}
\E_{\Q^\dr} \Big[ | R_{s, N}(Y_1^\dr + I_1(\dr)) | \Big] 
= \E \bigg[ |R_{s, N}(Y_1) | \frac{e^{-R_{s, N}(Y_1)}}{\ZZ_{s,N}} \bigg] < \infty.
\end{equation*}

Next, suppose  that 
\begin{align}
\E_{\Q^\dr} \bigg[ \int_0^1 \| \dr_t \|_{L^2_x}^2 dt \bigg] < \infty.
\label{Z6} 
\end{align}

\noi
Note that $\ZZ_{s,N} = \E [e^{-R_{s, N}(Y_1)} ]$. 
Then, 
by changing the measure
with \eqref{Z2a}, Jensen's inequality,
 and applying the decompositions \eqref{Z3} and \eqref{Z4} 
 of $X_t$ and $Y_t$ under $\Q^\dr$, we obtain
\begin{align} 
\begin{split}
-\log \ZZ_{s,N} 
& \leq \E_{\Q^\dr} \bigg[ R_{s, N}(Y_1) + \int_0^1 \jb{\dr_t,  dX_t} 
- \frac{1}{2} \int_0^1 \| \dr_t \|_{L^2_x}^2 dt \bigg] \\ 
&= \E_{\Q^\dr} \bigg[ R_{s, N}(Y_1^\dr + I_1(\dr)) 
+ \int_0^1 \jb{\dr_t,  dX^\dr_t} + \frac{1}{2} \int_0^1 \| \dr_t \|_{L^2_x}^2 dt \bigg].
\end{split}
\label{vform_1}
\end{align}

\noi
From \eqref{Z6}, 
we see that  the process $ \int_0^t \jb{\dr_{t'},  dX^\dr_{t'}}$ is a $\Q^\dr$-martingale and hence 
we conclude that 
\begin{align}
\E_{\Q^\dr} \bigg[ \int_0^1 \jb{\dr_t,  dX^\dr_t} \bigg] = 0.
\label{vform_1a}
\end{align}

\noi
Therefore, from \eqref{vform_1} and \eqref{vform_1a}, 
 we obtain \eqref{vform_aim}.

Next, we show that 
the infimum in \eqref{var} is indeed achieved	
for a special choice of drift. 
Given  $N \in \mathbb{N}$, define $\Q^N$ by the density
\begin{equation} \label{vform_2}
\frac{d\Q^N}{d\PP} = \frac{e^{-R_{s, N}(Y_1)}}{\ZZ_{s,N}}. 
\end{equation}

\noi
By the Brownian martingale representation theorem
(\cite[Proposition 1.6 in Chapter VIII]{RV}),
 there exists a drift $\wt{\dr}^N \in \Hc$ such that
\begin{equation} \label{vform_rep}
\frac{d\Q^N}{d\PP} = e^{ \int_0^1 \wt{\dr}^N_t dX_t - \frac{1}{2} \int_0^1 \| \wt{\dr}^N_t \|_{L^2_x}^2dt}.
\end{equation}

\noi
Then, from \eqref{vform_2} and\eqref{vform_rep}, we obtain
\begin{equation} \label{vform_3}
- \log \ZZ_{s,N} =  R_{s, N}(Y_1) + \int_0^1 \jb{\wt{\dr}^N_t,  dX_t} - \frac{1}{2} \int_0^1 \| \wt{\dr}^N_t \|_{L^2_x}^2dt.
\end{equation}

\noi
Taking expectations of \eqref{vform_3} with respect to $\Q^N$ 
and using the decompositions \eqref{Z3} and~\eqref{Z4} of $X_t$ and $Y_t$ under $\Q^N$, we obtain
\begin{equation} \label{vform_4}
- \log \ZZ_{s,N} =  \E_{\Q^N} 
\bigg[ R_{s, N}\big( Y_1^{\wt{\dr}^N} + I_1(\wt{\dr}^N)\big) + \int_0^1 \jb{\wt{\dr}^N_t,  dX^{\wt{\dr}^N}_t} 
+ \frac{1}{2} \int_0^1 \| \wt{\dr}^N_t \|_{L^2_x}^2dt \bigg].
\end{equation}

\noi
On the other hand, 
from  \eqref{vform_2} and \eqref{rent}, we have
\begin{align}
\E_{\Q^N} \bigg[ \log \frac{d\Q^N}{d\PP} \bigg] 
= \E \bigg[ \frac{e^{-R_{s, N}(Y_1)}}{\ZZ_{s,N}} \log \bigg( \frac{e^{-R_{s, N}(Y_1)}}{\ZZ_{s,N}} \bigg)  \bigg] < \infty.
 \label{vform_5}
\end{align}

\noi
In particular, it follows from \eqref{vform_5} and \eqref{Fol1} that
\[\E_{\Q^N} \bigg[ \int_0^1 \| \wt{\dr}^N_t \|_{L^2_x}^2 dt \bigg] < \infty.\]

\noi
This implies that the stochastic integral 
$ \int_0^t \jb{\wt{\dr}^N_{t'},  dX^{\wt{\dr}^N}_{t'}}$ is a $\Q^N$-martingale.
Therefore, from~\eqref{vform_4}, we obtain 
\begin{equation*} 
- \log \ZZ_{s,N} =  \E_{\Q^N} 
\bigg[ R_{s, N}\big( Y_1^{\wt{\dr}^N} + I_1(\wt{\dr}^N)\big) 
+ \frac{1}{2} \int_0^1 \| \wt{\dr}^N_t \|_{L^2_x}^2dt \bigg].
\end{equation*}

\noi
This completes the proof of Proposition \ref{PROP:var}.
\end{proof}

\begin{remark} \label{REM:BD}
\rm
The material presented above differs from \cite{BG} in the following ways: 
(i)  we do not need to introduce a time-dependent cutoff in the definition of $\{Y_t\}$
and (ii)   we do not need to use the stronger Bou\'e-Dupuis formula \cite{BD}:
\begin{align*} 
- \log \ZZ_{s,N} = \inf_{\dr \in \Ha} \E \left[ R_{s, N}(Y_1 + I_1(\dr)) + \frac{1}{2} \int_0^1 \| \dr_t \|_{L^2}^2dt \right].
\end{align*} 

\noi
See \cite{Ust} or Theorem 2 in \cite{BG} for further discussion. 

\end{remark}

\subsection{Exponential integrability}
 \label{SUBSEC:const2}
In this subsection, 
we present the proof of Proposition~\ref{PROP:QFT}
by studying the minimization problem \eqref{var}
in  Proposition \ref{PROP:var}. 
In particular, we show that the infimum in \eqref{var}
is bounded away from $- \infty$, uniformly in $N \in \N$.
Our strategy  is to 
use pathwise stochastic bounds on $Y_1^\dr$, uniform in the drift $\dr$
and 
use pathwise deterministic bounds on $I_1(\dr)$ independently of the drift
 (see Lemmas \ref{LEM:Y} and \ref{LEM:dr}).

We first state two lemmas on
the pathwise regularity estimates on $Y_1^\dr$ and $I_1(\dr)$.

\begin{lemma}  \label{LEM:Y}
Let $2\leq p < \infty$.
Then, we have
\begin{equation}
\sup_{\dr \in \Hc} \E_{\Q^\dr} 
\Big[\| D^s Y_1^\dr \|_{\C^{-\frac 12-\eps}}^p +  \| :\!(D^s Y_1^\dr)^2\!: \|_{\C^{-1-\eps}}^p   \Big]  <\infty
\label{Z7}
\end{equation}

\noi
for any $\eps > 0$.
Here, colons  denote  Wick renormalization.

\end{lemma}

\begin{proof}
Recall that
 $\{X_t^\dr\}$ under $\Q^\dr$ is a cylindrical Brownian motion in $L^2(\T^3)$
 for any $\dr \in \Hc$.
Thus, the supremum in \eqref{Z7}
is superfluous since the law of $Y_1^\dr = \J^{-s-1} X_1^\dr$ under $\Q^\dr$ is invariant under a change of drifts. 
In particular, we have
$\Law_{\Q^\ta}(Y_1^\ta) = \nu_s$. 
Then,~\eqref{Z7} follows
from the H\"older-Besov regularity of samples under $\nu_s$
and \eqref{rd1} in Proposition~\ref{PROP:random_distr}. 
\end{proof}

\begin{lemma}[Cameron-Martin drift regularity]  \label{LEM:dr} 
The drift term  $\dr \in \Hc$ has the  regularity  of the Cameron-Martin space $H^{s+1}(\T^3)$:
	\begin{equation}
	\| I_1(\dr) \|_{H^{s+1}}^2 \leq \int_0^1 \| \dr_t \|_{L^2}^2dt.
\label{Z8}
	\end{equation}
\end{lemma}

\begin{proof}
This is immediate from  Minkowski's integral inequality followed by Cauchy-Schwarz
inequality:
\begin{equation*}
\| I_1(\dr) \|_{H^{s+1}}= \bigg\| \int_0^1 \dr_t dt \bigg\|_{L^2} \leq \int_0^1 \| \dr_t \|_{L^2} dt \leq \left( \int_0^1 \| \dr_t \|_{L^2}^2 dt \right)^\frac{1}{2},
\end{equation*}

\noi
yielding \eqref{Z8}.
\end{proof}

We now present the proof of Proposition \ref{PROP:QFT},
 using Proposition \ref{PROP:var}. 
Fixing an arbitrary drift $\dr \in \Hc$, the quantity that we wish to bound from below is
\begin{equation}
\W_N(\dr) = \E_{\Q^\dr} 
\bigg[ R_{s, N}(Y_1^\dr + I_1(\dr)) + \frac{1}{2} \int_0^1 \| \dr_t \|_{L^2_x}^2 dt \bigg].	
\label{v_N0}
\end{equation}

\noi
Since the drift $\dr \in \Hc$ is fixed, 
we suppress the dependence on the drift $\dr$ henceforth 
and denote $Y = Y_1^\dr $ and $\Dr = I_1(\dr)$. 
From the definition \eqref{K1} of $R_{s, N}$, 
we have
\begin{align}
R_{s, N}(Y + \Dr)  & = 
\frac 32 \int_{\T^3} 
 : \!(D^s Y)^2 \!: (Y+\Dr)^2 + 2D^s Y D^s \Dr ( Y+\Dr)^2 + (D^s \Dr)^2 (Y+\Dr)^2 \notag \\
&\hphantom{X}
+ \frac 14 \int_{\T^3} (Y+\Dr)^4 . 
 \label{v_N}
\end{align}

\noi
The main strategy is to bound $\W_N(\dr)$ from below pathwise 
and independently of the drift by utilizing the positive terms:
\begin{equation}
\U_N(\dr) = \frac 32 \int (D^s \Dr)^2 \Dr^2 + \frac 14 \int \Dr^4 + \frac{1}{2} \int_0^1 \| \dr_t \|_{L^2_x}^2 dt .
\label{v_N1}
\end{equation}

\noi
In the following, we  state three lemmas, 
controlling the other terms appearing in~\eqref{v_N}.
The proofs of these lemmas follow from lengthy but straightforward computations
and are presented at the end of this section.
The first lemma handles the 
terms quadratic in $D^s Y$.

\begin{lemma}[Terms quadratic in $D^s Y$] \label{LEM:quad} 
Let $s > \frac 32$.
Then, given  $\dl>0$ sufficiently small, there exist
small $\eps > 0$ and  $c(\dl)>0$ such that
\begin{align}
\int_{\T^3} :\!(D^sY)^2\!: Y^2 
& \les \| :\!(D^s Y)^2\!: \|_{\C^{-1-\eps}}^2 + \| D^s Y \|_{\C^{-\frac 12 - \eps}}^4,
\label{quad_1}
 \\
\int_{\T^3} :\! (D^sY)^2\!: Y\Dr 
& \leq c(\dl) \Big( \| :\!(D^s Y)^2\!: \|_{\C^{-1-\eps}}^4 + \| D^s Y \|_{\C^{-\frac 12 - \eps}}^4 \Big) 
+ \dl \| \Dr \|_{H^{s+1}}^2,
\label{quad_2} \\
\int_{\T^3} :\!(D^sY)^2\!: \Dr^2 
& \leq c(\dl) \| :\!(D^s Y)^2\!: \|_{\C^{-1-\eps}}^4 + \dl \Big( \| \Dr \|_{H^{s+1}}^2 + \| \Dr \|_{L^4}^4 \Big).
\label{quad_3}
\end{align}

\end{lemma}

The next lemma handles  the 
terms linear in $D^s Y$.

\begin{lemma}[Terms linear in $D^s Y$] \label{LEM:lin} 
Let $s >  1$.
Then, given  $\dl > 0$ sufficiently small, 
there exist small $\eps > 0$,  $c(\dl) > 0$,  
and $p_j = p_j(\eps, s) > 1$, $j = 1,2$,  such that
\begin{align}
\int_{\T^3} D^sY D^s\Dr Y^2 
& \leq c(\dl)  \| D^s Y \|_{\C^{-\frac 12-\eps}}^6 + \dl \| \Dr \|_{H^{s+1}}^2, 
\label{lin_1}
 \\
\label{lin_2}
\int_{\T^3} D^sY D^s\Dr Y\Dr 
& \leq c(\dl) 
\Big( 1 + \| D^s Y \|_{\C^{-\frac 12-\eps}}  \Big)^{p_1}  
+\dl \Big( \| \Dr \|_{H^{s+1}}^2 + \| \Dr \|_{L^4}^4 \Big), 
 \\
\begin{split}
\int_{\T^3} D^sY D^s\Dr \Dr^2 
&\leq c(\dl) \Big(1+  \| D^s Y \|_{\C^{-\frac12-\eps}}\Big)^{p_2}  \\
&\hphantom{X}
+ \dl \Big( \| \Dr \|_{H^{s+1}}^2 + \| \Dr \|_{L^4}^4 
+ \| D^s \Dr \Dr \|_{L^2}^2\Big).
\end{split}
\label{lin_3}
\end{align}

\end{lemma}

Lastly, the third lemma controls 
the term quadratic in $D^s \Dr$.

\begin{lemma}[Term quadratic in $D^s \Dr$] \label{LEM:quadV}
Let $s > 1$.
Then, given $\dl > 0$,  there exist
small $\eps > 0$,  $c(\dl) > 0$,  
and $p = p(s, \eps) > 1$ 
such that
\begin{align}
\int_{\T^3} (D^s \Dr)^2 Y\Dr 
&\leq c(\dl) \| D^s Y \|_{\C^{-\frac 12-\eps}}^{p} 
+ \dl \Big( \| \Dr \|_{H^{s+1}}^2 + \| \Dr \|_{L^4}^4 +\| D^s \Dr \Dr \|_{L^2}^2 \Big).
\label{quadV}
\end{align}
\end{lemma}

The regularity restriction $s > \frac 32$ appears 
in controlling the terms quadratic in $D^s Y$.
We now prove Proposition \ref{PROP:QFT}, 
assuming Lemmas \ref{LEM:quad}, \ref{LEM:lin}, and \ref{LEM:quadV}.

First, note that the remaining terms left to treat in $\eqref{v_N}$
 are harmless. 
 The terms
$\int_{\T^3} (D^s \Dr)^2 Y^2$, 
$\int_{\T^3} Y^4$, and $\int_{\T^3} Y^2 \Dr^2$  are positive and thus can be discarded. 
The remaining two terms can be controlled by 
Young's inequality:
\begin{align*}
\int_{\T^3} Y^3 \Dr  +  \int_{\T^3} Y \Dr^3  \leq c(\dl) \| Y \|_{L^4}^{4} + \dl \| \Dr \|_{L^4}^4
\end{align*}

\noi
for any $\dl > 0$.
We now apply
 the regularity estimates of Lemmas \ref{LEM:Y} and \ref{LEM:dr} to the bounds obtained in 
 Lemmas \ref{LEM:quad}, \ref{LEM:lin}, and \ref{LEM:quadV}, and the bounds on the harmless terms.
 Then, from \eqref{v_N0}, \eqref{v_N}, and \eqref{v_N1}, 
 we  conclude that, by choosing $\dl > 0$ sufficiently small, there exists finite $C = C(\dl)>0$ such that
\begin{align*}
\sup_{N \in \mathbb{N}} \sup_{\dr \in \Hc} \W_N(\dr) 
\geq 
\sup_{N \in \mathbb{N}} \sup_{\dr \in \Hc}
\Big\{ -C(\dl) + \frac{1}{4}\U_N(\dr)\Big\}
 \geq - C(\dl)
>-\infty.
\end{align*}

\noi
Therefore, 
 by Proposition \ref{PROP:var}, this proves Proposition \ref{PROP:QFT}
 (when $p = 1$).

\medskip

In the remaining part of this section, 
we present the proofs of 
 Lemmas \ref{LEM:quad}, \ref{LEM:lin}, and~\ref{LEM:quadV}.

\begin{proof}[Proof of Lemma \ref{LEM:quad}]
	
By duality $\eqref{dual}$ and the algebra property $\eqref{alge}$, we have
\begin{equation*}
\text{LHS of } \eqref{quad_1}
\leq  \| :\!(D^s Y)^2\!: \|_{\B^{-1-2\eps}_{1,1}} \| Y \|_{\C^{1+2\eps}}^2.
\end{equation*}

\noi
Then, by choosing $\eps > 0$ sufficiently small, 
\eqref{quad_1} follows from the trivial embeddings \eqref{embed}
and  Cauchy's inequality, provided that $s > \frac 32$.
	
By  duality $\eqref{dual}$ and the fractional Leibniz rule  $\eqref{prod}$, 
we have
\begin{align*} 
\text{LHS of } \eqref{quad_2}
&\les \| :\!(D^sY)^2\!: \|_{\B^{-1 - 2 \eps}_{\infty,2}}
\| Y\Dr \|_{\B^{1+2 \eps}_{1,2}} \\ 
&\les \| :\!(D^s Y)^2 \!: \|_{\C^{-1-\eps}} 
\Big( \| Y \|_{\B^{1+2 \eps}_{2,2}}\| \Dr \|_{L^2} + \| Y \|_{L^2}\| \Dr \|_{\B^{1+2\eps}_{2,2}} \Big).
\end{align*}

\noi
Then, by choosing $\eps > 0$ sufficiently small, \eqref{quad_2} follows from 
 \eqref{embed} 
and Young's inequality,
 provided that $s > \frac 32$.
	
Lastly, proceeding as above with \eqref{dual}
and  \eqref{prod}, we have
\begin{align*} 
\text{LHS of } \eqref{quad_3}
\les \| :\!(D^s Y)^2\!: \|_{\B^{-1-2\eps}_{\infty,2}}
\| \Dr \|_{\B^{1+2\eps }_{2,2}} \| \Dr \|_{L^2}.
\end{align*}

\noi
Then,  $\eqref{quad_3}$ follows from  \eqref{embed}, $L^4 (\T^3)  \hra L^2(\T^3)$, and Young's inequality. 		
\end{proof}

Next, we present the proof of  Lemma \ref{LEM:lin}.
The main idea
is to use
(i) $ \| \Dr \|_{H^{s+1}}$ for controlling derivatives on $\Dr$
and  (ii) $\| \Dr \|_{L^4}$ 
and  $\| D^s \Dr \Dr \|_{L^2}$ 
for controlling homogeneity of~$\Dr$.

\begin{proof}[Proof of Lemma \ref{LEM:lin}]

By duality \eqref{dual} and the fractional Leibniz rule \eqref{prod} with \eqref{embed}, 
we have
\begin{align*} 
\text{LHS of } \eqref{lin_1}
 &\les \| D^s Y \|_{\B^{-\frac12-2\eps}_{\infty, 2}} \| D^s \Dr Y^2 \|_{\B^{\frac12+2\eps}_{1,2}} \\
& \les  \| D^s Y \|_{\C^{-\frac12-\eps}} 
\Big( \| Y^2 \|_{\B^{\frac 12+2 \eps}_{2,2}} \| D^s \Dr \|_{L^2} 
+ \| Y^2 \|_{L^2} \| D^s \Dr \|_{B^{\frac 12+2\eps}_{2,2}} \Big)\\
& \leq  \| D^s Y \|_{\C^{-\frac 12-\eps}} \| Y \|_{\C^{\frac 12+3\eps}}^2 \| \Dr \|_{H^{s+1}}.
\end{align*}

\noi
Then, by choosing $\eps > 0$ sufficiently small, 
\eqref{lin_1} follows from Cauchy's inequality, 
provided that $s > 1$.

By duality \eqref{dual} and the fractional Leibniz rule \eqref{prod} with \eqref{embed}
and \eqref{alge}, 
we have
\begin{align*}
\text{LHS of } \eqref{lin_2}
&\les \| D^s Y \|_{\B^{-\frac 12-2\eps}_{\infty,2}} \| D^s \Dr Y \Dr \|_{\B^{\frac 12+2\eps}_{1,2}} \\
& \les \| D^s Y \|_{\C^{-\frac 12-\eps}} 
\Big( \| Y\Dr\|_{\B^{\frac 12+2\eps}_{2,2}} \| D^s \Dr \|_{L^2} 
+ \| Y\Dr \|_{L^2} \| D^s \Dr \|_{\B^{\frac 12+2\eps}_{2,2}} \Big) \\
&
=:  T_1 + T_2.
\end{align*}

\noi
By  H\"older's inequality and \eqref{embed}, we have
\begin{align} 
\begin{split}
T_2 
& \les \| D^s Y \|_{\C^{-\frac 12-\eps}} \| Y \|_{L^4} \| \Dr \|_{H^{s+1}} \| \Dr \|_{L^4}  \\
& \les \| D^s Y \|_{\C^{-\frac 12-\eps}}^2 \| \Dr \|_{H^{s+1}} \| \Dr \|_{L^4}
\end{split}
\label{T2}
\end{align}

\noi
for $s > \frac 12$ and small $\eps > 0$.

By \eqref{prod}, \eqref{embed},  and the interpolation \eqref{interp},
we have
\begin{align*}
\| Y\Dr \|_{\B^{\frac 12+2 \eps}_{2,2}} 
& \les
 \| Y \|_{\B^{\frac 12+2\eps}_{\infty,2}} \| \Dr \|_{L^2}
+  \| Y \|_{L^\infty} \| \Dr \|_{\B^{\frac 12+2\eps}_{2,2}} \\
& \les \| Y \|_{\C^{\frac 12 + 3\eps} }\| \Dr \|_{H^{\frac 12+2\eps}}\\
& \les \| Y \|_{\C^{\frac 12 + 3\eps} }\| \Dr \|_{H^{s+1}}^\g \|\Dr\|_{L^2}^{1-\g}
\end{align*}

\noi
for some $\g = \g(s, \eps) \in (0, 1)$.
Thus, we have
\begin{align} 
\begin{split}
T_1
& \les \| D^s Y \|_{\C^{-\frac 12-\eps}}^2 
\| \Dr \|_{H^{s+1}}^{1+\g} \|\Dr\|_{L^4}^{1-\g}
\end{split}
\label{T1}
\end{align}

\noi
for $s > 1$ and small $\eps > 0$.
Hence, noting that $\frac 12 + \frac 14 < 1$ and $\frac{1+\g}{2} + \frac{1 - \g}{4} < 1$
for $\g \in (0, 1)$, 
the desired estimate \eqref{lin_2} follows
from applying Young's inequality to \eqref{T2} and \eqref{T1}.

Finally, we consider $\eqref{lin_3}$. 
By  \eqref{dual} and \eqref{prod}
with \eqref{embed}, 
we have
\begin{align*}
\text{LHS of } \eqref{lin_3} 
 &\les \| D^s Y \|_{\B^{-\frac12-2\eps}_{\infty, 1}} \| D^s \Dr \Dr^2 \|_{\B^{\frac 12 +2\eps}_{1,\infty}} \\
& \les \| D^s Y \|_{\C^{-\frac12-\eps}} 
\Big( \| D^s \Dr \Dr \|_{L^2} \| \Dr \|_{\B^{\frac 12+2\eps}_{2,\infty}} 
+ \| D^s\Dr \Dr\|_{\B^{\frac12+2\eps}_{2, \infty}} \| \Dr \|_{L^2} \Big) \\
& =: T_3 + T_4.
\end{align*}

\noi
By the interpolation \eqref{interp} with $L^4 (\T^3) \hra L^2(\T^3)$, 
there exists $\g_1  = \g_1(s , \eps) \in (0, 1)$ such that
\begin{align*}
T_3 
& \les \| D^s Y \|_{\C^{-\frac 12-\eps}} 
\| D^s \Dr \Dr \|_{L^2} \| \Dr \|_{H^{s+1}}^{\g_1} \| \Dr \|_{L^4}^{1-\g_1}.
\end{align*}

\noi
Noting that
$\frac 12 + \frac{\g_1}{2} + \frac {1-\g_1}{4} < 1$, 
we can apply Young's inequality to 
bound the contribution from $T_3$ by the right-hand side of \eqref{lin_3}.

It remains to estimate $T_4$.
By the interpolation \eqref{interp} and \eqref{prod}, 
we have 
\begin{align}
\begin{split}
\| D^s \Dr \Dr \|_{H^{\frac 12+2\eps}}\| \Dr \|_{L^2} 
 &\les \| D^s \Dr \Dr \|_{H^{1}}^{\g_2} \| D^s \Dr \Dr \|_{L^2}^{1-\g_2} \| \Dr \|_{L^2} \\
& \les \Big( \| D^s \Dr \|_{\B^{1}_{2,2}} \| \Dr \|_{L^\infty} 
+ \| D^s \Dr \|_{L^6} \| \Dr \|_{\B^{1}_{3, 2}} \Big)^{\g_2}  \\ 
&\hphantom{X}
 \times \| D^s \Dr \Dr \|_{L^2}^{1-\g_2} \| \Dr \|_{L^4},
\end{split}
 \label{lin3_t1}
\end{align}

\noi
where $\g_2 = \g_2(\eps) \in(0, 1)$ is given by 
\begin{align}
\g_2 = \frac 12 + 2\eps.
\label{g2}
\end{align}

\noi
By Sobolev's inequality
and the interpolation \eqref{interp} (with $s > \frac 12$), 
we have
\begin{equation}
 \| D^s \Dr \|_{\B^{1}_{2,2}} \| \Dr \|_{L^\infty} 
+ \| D^s \Dr \|_{L^6} \| \Dr \|_{\B^{1}_{3, 2}} 
 \les \| \Dr \|_{H^{s+1}} \| \Dr \|_{H^{\frac 32 + \eps}} 
 \les \| \Dr \|_{H^{s+1}}^{1 +\g_3} \| \Dr \|_{L^4}^{1-\g_3},
 \label{lin3_t1a}
\end{equation}

\noi
where $\g_3 = \g_3(s, \eps) \in(0, 1)$ is given by 
\begin{align}
\g_3 = \frac{3 + 2\eps}{2(s+1)}.
\label{g3}
\end{align}

\noi
Combining \eqref{lin3_t1} and \eqref{lin3_t1a}, we obtain
\begin{equation*} 
T_4 \les \| D^s Y \|_{\C^{-\frac 12-\eps}} 
\| \Dr \|_{H^{s+1}}^{ \g_2(1+\g_3)} 
\| D^s \Dr \Dr \|_{L^2}^{1-\g_2} \| \Dr \|_{L^4}^{1 +\g_2(1-\g_3)}.
\end{equation*}

\noi
From \eqref{g2} and \eqref{g3}, we observe that 
\[\frac{ \g_2(1+\g_3)}{2}
+ \frac{1-\g_2}{2} 
+ \frac{1 +\g_2(1-\g_3)}{4} < 1, \]

\noi
provided that $s > \frac 12$ and $\eps > 0$ is sufficiently small.
Therefore, 
we can apply Young's inequality to 
bound the contribution from $T_4$ by the right-hand side of \eqref{lin_3}.
This completes the proof of Lemma \ref{LEM:lin}.
\end{proof}

We conclude this section by presenting the proof of Lemma \ref{LEM:quadV}.

\begin{proof}[Proof of Lemma \ref{LEM:quadV}]
By Cauchy's inequality, we have
\begin{equation}
\int_{\T^3} (D^s \Dr)^2 Y \Dr  \leq c(\dl) \int_{\T^3} (D^s \Dr)^2 Y^2  + \dl \| D^s \Dr \Dr\|_{L^2}^2.
 \label{quadV_1}
\end{equation}

\noi
By H\"older's and Sobolev's inequalities
followed by 
the interpolation \eqref{interp}
with \eqref{embed} and~\eqref{alge}, we have
\begin{align}
\begin{split}
\int_{\T^3} (D^s \Dr)^2 Y^2 
&\les \| D^s \Dr \|_{L^3}^2 \| Y^2 \|_{L^3} \les \|  \Dr \|_{H^{s+ \frac 12}}^2 \| Y^2 \|_{ H^{\frac 12}} \\
& \les \|  \Dr \|_{H^{s+1}}^{2\g} \|  \Dr \|_{L^2}^{2(1-\g)} \| Y^2 \|_{\C^{\frac12 +\eps}} 
 \\
& \les \|  \Dr \|_{H^{s+1}}^{2\g} \|  \Dr \|_{L^4}^{2(1-\g)} \| D^s Y \|_{\C^{-\frac12 -\eps}}^2 
\end{split}
\label{quadV_2}
\end{align} 

\noi
for some $\g = \g(s) \in (0, 1)$, provided that $s > 1$ and $\eps > 0$ is sufficiently small.
Noting that $\frac{2\g}{2} + \frac{2(1-\g)}{4} < 1$, 
\eqref{quadV} follows from 
\eqref{quadV_1}, 
\eqref{quadV_2}, and 
Young's inequality.
\end{proof}

\section{Renormalized energy estimate} 
\label{SEC:ENERGY}

Recall from  \eqref{H8} that 
\begin{equation*}
\partial_t E_{s,N} (\pi_N \Phi_N(t) (\vec{u}) )\Big|_{t=0}  = F_1(\vec{u}_N) + F_2(\vec{u}_N) + F_3(\vec{u}_N), 	
\end{equation*}

\noi
where $\u_N = (u_N, v_N)$ and 
\begin{align*}
F_1(\vec{u}_N) &= 3 \int_{\T^3}  \DD  v_N  u_N,  \\
F_2(\vec{u}_N) &= \sum_{\substack{ |\al|+|\be|+|\g|=  s \\
|\al|,|\be|,|\g|<s
}}
c_{\al,\be,\gamma}
\int_{\T^3}
D^sv_N\cdot \dd^\al u_N \cdot \dd^\be u_N \cdot \dd^\g u_N , \\
F_3(\vec{u}_N) &= \bigg(\int_{\T^3}  u_N\bigg) \bigg( \int_{\T^3} v_N \bigg).
\end{align*}

\begin{proposition}
\label{PROP:energyestimate}
Let $s \geq 4$ be an even integer.  Then,
there exist $\s < s - \frac 12$ sufficiently close to $s - \frac 12$ and small $\eps >0$
such that 
\begin{align}
\bigg|\,\partial_t E_{s,N} (\pi_N \Phi_N(t) (\vec{u}) )\Big|_{t=0} \bigg|
\leq \big(1 +\| \vec{u}_N \|^2_{\H^\s}\big) F(\vec{u}_N) , 
\label{X1}
\end{align}

\noi
where
\begin{align*}
\begin{split}
F(\vec{u}_N)  = 1  &  + \| \DD \|_{\C^{-1-\varepsilon}} \\
& +
 \sup_{\substack{|k|=s-1 \\ |\alpha|=s}}\| \partial^\kk v_N \, \partial^\alpha u_N \|_{\C^{-1-\varepsilon}}
 + 
\sup_{\substack{|k|=s-1 \\ |\alpha| \leq s-1}}\| \partial^\kk v_N \, \partial^\alpha u_N \|_{\C^{- \frac 12 - \varepsilon}}.
\end{split}
\end{align*}

\end{proposition}

\begin{proof}
In the following, we prove \eqref{X1} uniformly in $N \in \N$.
Thus,  we drop the $N$-dependence and write $Q_s(u)$ for $\DD$.

First, note that  the estimate for $F_3$ follows trivially from Cauchy-Schwarz inequality. 
Next, we treat $F_1$.
By  duality \eqref{dual} and the fractional Leibniz rule \eqref{prod}, we have
\begin{align}
\begin{split}
\int_{\mathbb{T}^3} \DDD u v 
& \lesssim \| \DDD \|_{\mathcal{C}^{-1-\varepsilon}}
 \|  u  v \|_{\mathcal{B}^{1+\varepsilon}_{1,1}}	\\
& \lesssim \| \DDD \|_{\mathcal{C}^{-1-\varepsilon}} \| u \|_{H^\sigma} \| v \|_{H^{\sigma - 1}}, 
\end{split}
\label{est0}
\end{align}

\noi
provided that $\s > 2 + \eps$.
This is guaranteed by choosing $\s$ sufficiently close to $s - \frac 12$, when $s > \frac 52$.

It remains to consider $F_2$.
By integration by parts, it suffices to  consider terms of the form:
\begin{equation*}
\int_{\mathbb{T}^3} \partial^\kk v \, \partial^\alpha u \, \partial^\beta u \, \partial^\gamma u,
\end{equation*}

\noi
where $|\kk| = s-1$, $\max(\alpha, \beta, \gamma) \leq s$,
 and $|\alpha| + |\beta| + |\gamma| = s+1$. Without loss of generality, we assume 
 that $|\alpha| \geq |\beta| \geq |\gamma|$. 
The idea is to group the low regularity terms ($\partial^\kk v$ and $\partial^\alpha u$) and treat them as one piece.

First, let us assume that $|\alpha|=s$.
In this case, we have  $|\beta| = 1$ and $|\gamma|=0$.
By duality~\eqref{dual} and the fractional Leibniz rule \eqref{prod}, we have
\begin{align} 
\bigg|\int_{\mathbb{T}^3} \partial^\kk v  \partial^\alpha u \, \partial u  \, u \bigg|
\lesssim \| \partial^\kk v \, \partial^\alpha u \|_{\mathcal{C}^{-1-\varepsilon}} \| \partial u \, u \|_{\mathcal{B}^{1+\varepsilon}_{1,1}} \lesssim \| \partial^\kk v \, \partial^\alpha u \|_{\mathcal{C}^{-1-\varepsilon}} \| u \|_{H^{\sigma}}^2,
\label{est1}
\end{align}

\noi
provided that $\s > 2 +\eps$.
By choosing $\eps > 0$ sufficiently small, 
we can guarantee this condition if  $s > \frac{5}{2}$.

This leaves the case $|\alpha| \leq s-1$.
Noting  that $|\beta| \leq \frac{s+1}{2}$ and $|\gamma| \leq \frac{s+1}{3}$
(under $|\alpha| \geq |\beta| \geq |\gamma|$), 
we see that  $\partial^\beta u, \partial^\gamma u \in H^{\frac{1}{2}+\varepsilon}(\T^3)$
for $s > 3$.
Thus, by duality \eqref{dual} and the fractional Leibniz rule~\eqref{prod}, we have:
\begin{equation} \label{est2}
\bigg|\int_{\mathbb{T}^3} \partial^\kk v \, \partial^\alpha u \, \partial^\beta u \dd^\g  u \bigg|
\lesssim \| \partial^\kk v \, \partial^\alpha u \|_{\mathcal{C}^{-\frac{1}{2}-\varepsilon}} \| \partial^\beta u \, u \|_{\mathcal{B}^{\frac{1}{2}+\varepsilon}_{1,1}} \lesssim \| \partial^\kk v \, \partial^\alpha u \|_{\mathcal{C}^{-\frac{1}{2}-\varepsilon}} \| u \|_{H^{\sigma}}^2.
\end{equation}

\noi
This completes the proof of Proposition \ref{PROP:energyestimate}.
\end{proof}

\begin{remark}\rm
The restriction $s > 3$ in the last case appears only when $|\be| = \frac{s+1}{2}$.
In fact,  when $|\be| \leq \frac s2$, 
the estimate \eqref{est2} holds true for $s > 2$.
On the other hand, when $|\be| = \frac{s+1}{2}$, 
we must have $|\al| = |\be| = \frac{s+1}{2}$.
In this case, by applying dyadic decompositions
and working with the Littlewood-Paley pieces $\P_{j_2} \dd^\al u\, \P_{j_3} \dd^\be u$, 
we can move half a derivative from the third factor to the second factor,
thus showing that a slight variant of \eqref{est2} holds for $s > 2$.
Therefore, the estimates \eqref{est0} and \eqref{est1}  on $F_1$ and $F_2$
impose the regularity restriction~$s > \frac 52$.

\end{remark}

Finally, we conclude this paper 
by establishing 
the renormalized energy estimate (Proposition \ref{PROP:Nenergy}).

\begin{proof}[Proof of Proposition \ref{PROP:Nenergy}]
The renormalized energy estimate \eqref{XX1}
 follows from Proposition~\ref{PROP:energyestimate}, the cutoff in the $\H^\s$-norm, and 
Proposition  \ref{PROP:random_distr}
 with \eqref{XX2}, 
 controlling $F(\vec{u}_N)$.
\end{proof}

\begin{ackno}\rm

T.S.G.~was supported by ESPRC as part of the MASDOC CDT at the University of Warwick, Grant No.~EP/HO23364/1. 
T.O.~was supported by the European Research Council (grant no.~637995 ``ProbDynDispEq''
and grant no.~864138 ``SingStochDispDyn"). 
N.T.~was supported by the ANR grant ODA (ANR-18-CE40-0020-01).
H.W.~was supported by the Royal Society through the University Research Fellowship UF140187.
T.S.G.~and H.W.~would like to thank the Isaac Newton Institute for Mathematical Sciences for support and hospitality during the program: ``Scaling limits, rough paths, quantum field theory'' when work on this paper was undertaken.  
This work was supported by EPSRC Grant Number EP/R014604/1.

\end{ackno}

\end{document}